\documentclass[]{amsart}

\usepackage{amssymb}
\usepackage{amsmath}
\usepackage{mathtools}
\usepackage{graphicx}
\usepackage{epstopdf}
\usepackage[sf,SF]{subfigure}
\usepackage{xcolor}
\usepackage{color}
\usepackage{sidecap}
\usepackage{pagecolor}
\usepackage{tikz}
\usepackage{colortbl} 
\usepackage{pdfpages}
\usepackage{hyperref}

\numberwithin{equation}{section}

\newtheorem{theorem}{Theorem}[section]
\newtheorem{lemma}[theorem]{Lemma}
\newtheorem{proposition}[theorem]{Proposition}

\newtheorem{corollary}[theorem]{Corollary}
\newtheorem{definition}[theorem]{Definition}
\newtheorem{remark}[theorem]{Remark}
\newtheorem{example}[theorem]{Example}

\newcommand{\Dx}{ {\Delta x } }
\newcommand{\Dt}{ {\Delta t } }

\newcommand{\R}{\mathbb{R}}

\newcommand{\Z}{\mathbb{Z}}
\newcommand{\N}{\mathbb{N}}
\newcommand{\p}{\mathbb{P}}

\newcommand{\D}{\mathcal{D}}

\renewcommand{\epsilon}{\varepsilon}

\newcommand{\RN}[1]{%
  \textup{\uppercase\expandafter{\romannumeral#1}}%
}

\newcommand*{\mailto}[1]{\href{mailto:#1}{\nolinkurl{#1}}}
\newcommand{\non}{\nonumber \\}

\renewcommand{\d}{\mathrm{d}}

\begin{document}
\title{Numerical conservative solutions of the Hunter--Saxton equation}


\allowdisplaybreaks

\author[K. Grunert]{Katrin Grunert }
\address{Department of Mathematical Sciences\\ Norwegian University of Science and Technology\\Trondheim\\Norway}

\author[A. Nordli]{Anders Nordli}
\address{Department of Automation and Process Engineering\\ UiT - The Arctic University of Norway\\ Troms\o \\ Norway}

\author[S. Solem]{Susanne Solem}

\thanks{Research supported by the grants {\it Waves and Nonlinear Phenomena (WaNP)} and {\it Wave Phenomena and Stability --- a Shocking Combination (WaPheS)} from the Research Council of Norway.}  

\begin{abstract}
In the article a convergent numerical method for conservative solutions of the Hunter--Saxton equation is derived. The method is based on piecewise linear projections, followed by evolution along characteristics where the time step is chosen in order to prevent wave breaking. Convergence is obtained when the time step is proportional to the square root of the spatial step size, which is a milder restriction than the common CFL condition for conservation laws.
\end{abstract}

\maketitle

\section{Introduction}
\noindent The Hunter--Saxton (HS) equation is given by
\begin{equation}
\label{eq:HS}
u_t(t,x)+uu_x(t,x) = \frac 12\int_{-\infty}^x u_x^2(t,y)\:\d y - \frac 14\int_{-\infty}^\infty u_x^2(t,y)\:\d y.
\end{equation}
It was derived, in differentiated form, from the nonlinear variational wave equation $\psi_{tt}-c(\psi)(c(\psi)\psi_x)_x = 0$ as an asymptotic model of the director field of a nematic liquid crystal \cite{HS}. Furthermore, the Hunter--Saxton equation is the high frequency limit of the Camassa--Holm equation \cite{DP}. It is completely integrable \cite{HZ} and can be interpreted as a geodesic flow \cite{KM}.

Another main property is that weak solutions are not unique, see e.g. \cite{HZ2,HZ3}. The main reason being the following:  
Solutions of \eqref{eq:HS} may experience wave breaking in finite time, i.e., $u_x\rightarrow -\infty$ pointwise while the energy $\|u_x(t,\cdot)\|_2$ remains uniformly bounded and the solution $u$ stays continuous. Furthermore, a finite amount of energy is concentrated on a set of measure zero. 

We illustrate wave breaking with an example by considering a peakon solution --- a soliton-like solution that is continuous and piecewise linear in space. It is not a classical solution. Indeed the function is not differentiable at the break points between the linear segments.
\begin{example}[Wave breaking for peakons]
	\label{ex:peakon}
	A particular peakon solution that illustrates wave breaking is given by
	\begin{equation*}
	u(t,x) =
	\begin{cases}
	1-\frac12t,\qquad & x<-1+t-\frac14t^2,\\
	-\frac1{1-\frac12t}x,\qquad & -1+t-\frac14t^2\leq x\leq 1-t+\frac14t^2,\\
	-1+\frac12t,\qquad & 1-t+\frac14t^2<x,
	\end{cases}
	\end{equation*}
	with $0\leq t < 2$. Note that for $t<2$,
	\begin{equation*}
	\left(u_x(t,x)^2\right)_t + \left(u(t,x)u_x(t,x)^2\right)_x = 0,
	\end{equation*}
	that is $\| u_x(t,\cdot)\|_2$ is a conserved quantity. As $t\rightarrow 2^-$, we see that $u_x(t,0)\rightarrow-\infty$ while the interval $[-1+t-\frac14t^2, 1-t+\frac14t^2]$ shrinks to a single point (see Figure \ref{fig:cusp1}). One can check that the function $u$ remains uniformly bounded and uniformly H\" older continuous with exponent $\frac12$ on $[0,2]\times\mathbb{R}$.
\end{example}
\begin{figure}
	\centering
	\subfigure{
	\includegraphics[width=0.38\textwidth]{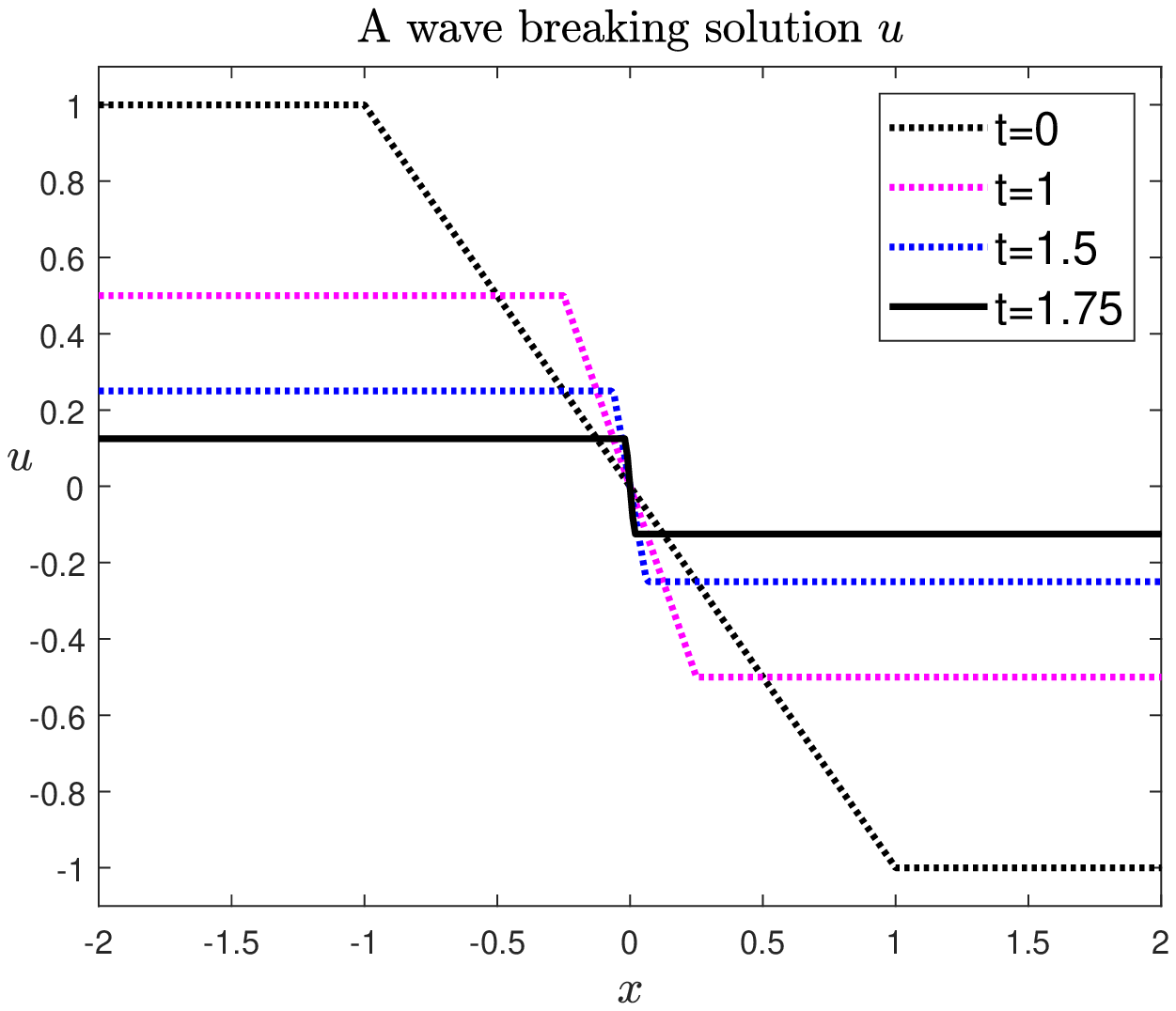}}\hspace{0.5em}
	\subfigure{	
	\includegraphics[width=0.38\textwidth]{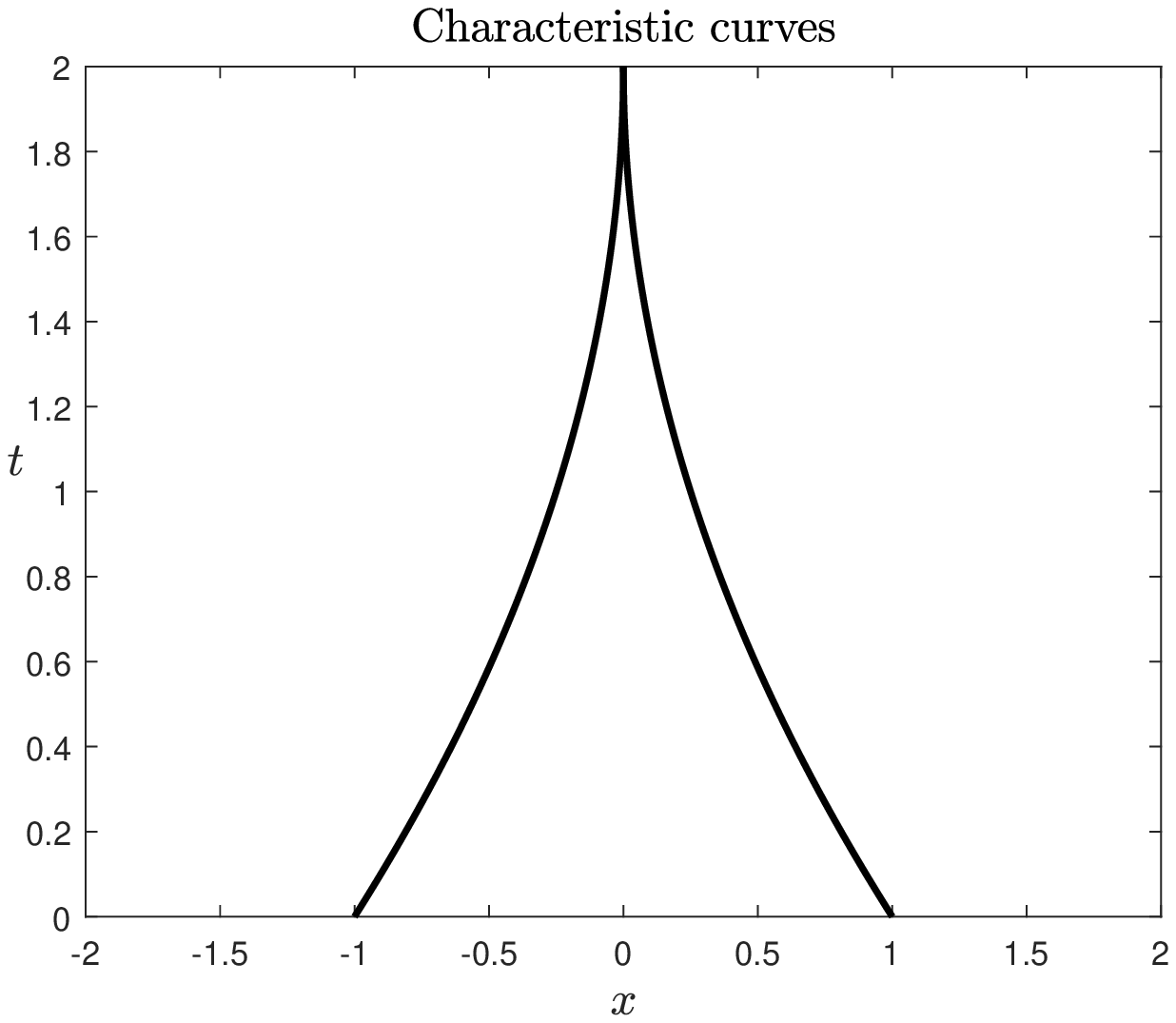}}
	\caption{The solution in Example \ref{ex:peakon} as $t$ tends to 2 (left). The (characteristic) curves describing the position of the break points in Example \ref{ex:peakon} (right).}
	\label{fig:cusp1}
\end{figure}

It is possible to extend weak solutions past wave breaking in various ways, see \cite{BC,BHR,BZZ,GN,N}. One could ignore the part of the solution that blows up. That amounts to continuing the solution in Example \ref{ex:peakon} as $u(t,x) = 0$ for all $t\geq 2$. Such solutions are called (energy) dissipative and are unique \cite{CJ,D}. A different approach is to ``reinsert'' the energy after wave breaking to get (energy) conservative solutions. To extend the solution in Example \ref{ex:peakon} as a conservative solution we let the formula defining $u$ hold for $t \geq 2$ as well. Uniqueness of conservative solutions is only known in several special cases \cite{ZZ,ZZ2}. The different solution concepts mimic the ones for some closely related equations: the Camassa--Holm equation \cite{CH}, the nonlinear variational wave equation \cite{SX}, and various generalizations of these equations. 

From now on we focus on weak solutions that preserve the energy, that is conservative solutions. 
It has been shown in \cite{BHR} that there exists a Lipschitz continuous semigroup of weak conservative solutions to \eqref{eq:HS}. Existence of solutions is proved using Lagrangian coordinates and characteristics. Note that the curves describing the position of the break points in Example \ref{ex:peakon} are examples of characteristic curves.

To prolong the solution past wave breaking and to attempt to overcome the non-uniqueness of weak solutions past wave breaking, we include the cumulative energy $F$ as part of the solution. The HS equation is then reformulated as
\begin{subequations}
	\label{eq:HSsys}
	\begin{align}
	u_t+uu_x &= \frac 12F-\frac 14F_\infty,\\
	F_t+uF_x &= 0,
	\end{align}
\end{subequations}
with appropriate initial conditions and the conditions 
\begin{subequations}\label{cond}
\begin{align} \label{cond:1}
&F(t,x)=\mu(t,(-\infty,x)) \text{ for some positive, finite Radon measure } \mu(t,\cdot), \\ \label{cond:2}
&\lim_{x\rightarrow\infty}F(t,x) = F_\infty, \\ \label{cond:4}
&\int_a^b u_x^2(t,x) \:\d x=\mu_{ac}(t,(a,b)),
 \end{align}
 \end{subequations}
where $\mu_{ac}(t,\cdot)$ is the absolutely continuous part of $\mu(t,\cdot)$. A closer look at the imposed conditions reveals that one challenge is to find a numerical method that respects condition \eqref{cond:4}. The key to overcome this difficulty is to consider \eqref{eq:HSsys} with the slightly more general conditions \eqref{cond:1}, \eqref{cond:2}, and 
\begin{equation}\label{cond:3}
 \int_a^b u_x^2(t,x)\:\d x \leq F(t,b)-F(t,a).
\end{equation}
This new system is a reformulation of the so-called two-component Hunter--Saxton (2HS) system, which not only generalizes the HS equation, but can also be studied using the same methods and ideas, see \cite{GN} and \cite{N}. Moreover, every conservative solution to the HS equation can be approximated by smooth solutions of the 2HS system. Of particular interest for us is the fact that if $u$ and $F$ are piecewise linear and continuous on some interval $[c,d]$ and 
\begin{equation*}
\int_a^b u_x^2(t,x)\:\d x<F(t,b)-F(t,a) \quad \text{ for all }c\leq a <b\leq d,
\end{equation*} 
then this property will be preserved along characteristics and no wave breaking takes place. Furthermore note that applying a piecewise linear projection operator to pairs $(u,F)$ satisfying \eqref{cond:4} yields pairs $(\tilde u, \tilde F)$ satisfying \eqref{cond:3}. Thus using the method of characteristics and piecewise linear projection operators as building blocks for a numerical method seems to be a good choice.

\subsection{Numerical methods for the Hunter--Saxton equation}

Despite receiving a considerable amount of attention theoretically, relatively little numerical work has been done on the Hunter--Saxton equation. In \cite{HKR} a finite difference method was constructed and proved to converge to dissipative solutions. In \cite{XS} and \cite{XS2} discontinuous Galerkin methods were introduced, followed by a convergence proof in the dissipative case but not in the conservative case. More recently, a geometric numerical integrator, based on the complete integrability of \eqref{eq:HS}, was introduced and studied in \cite{MCFM}. The method seems to converge to the conservative solutions, but no proof was presented. In \cite{ST} a difference method that converges for smooth solutions of a modified Hunter--Saxton equation in the periodic setting was introduced. The analysis in \cite{ST} does not apply to our setting since the method relies crucially on the modification of the equation, and even for \eqref{eq:HS} the periodic case and the real line case are essentially different \cite{Y}.
\\[1em]
In this paper, we contribute to this line of research by introducing a convergent numerical method for the conservative solutions of the Hunter--Saxton equation \eqref{eq:HSsys}. The method, introduced at the beginning of Section~\ref{sec:IntNum}, is inspired by Godunov-type methods for conservation laws and is based on piecewise linear projections, followed by evolution along characteristics forward in time. As for finite difference (and volume) schemes for conservation laws, where one limits the time step $\Dt$ to prevent shocks from occurring, we limit the size of $\Dt$ to prevent wave breaking (see \eqref{eq:cfl}). In contrast to the situation for conservation laws, we get the improved bound $\Dt\leq C\sqrt{\Dx}$ for some $C$ that depends on the initial data. 

After establishing some a priori bounds of the numerical solutions in Section~\ref{sec:aprioribounds}, we show in Section~\ref{sec:conv} that the numerical approximation converges with a rate of $\mathcal{O}(\sqrt{\Dx})$ to the unique solution of \eqref{eq:HSsys} whenever the solution is Lipschitz continuous. We also prove the existence of a convergent subsequence of the proposed numerical method in the general case, which converges to a weak solution preserving $F$. Unfortunately, the present lack of a satisfactory uniqueness theory for conservative solutions of \eqref{eq:HS} prevents us from guaranteeing that the sequence as a whole converges to the unique conservative solution. However, we perform numerical experiments towards the end of the paper, see Section~\ref{sec:numex}, showing that the numerical approximations seem to converge towards the desired solutions also in the case of non-Lipschitz solutions.

\section{Numerical conservative solutions of the Hunter--Saxton equation}\label{sec:IntNum}
\noindent For the (to be defined) numerical solutions to approximate conservative solutions of the HS equation, we will require that they mimic certain aspects of these solutions. In particular, we will design a method such that the numerical approximations are pairs $(u,F)$ in a suitable function space $\D$, which resembles the one used for the 2HS system in \cite{N}: 
\begin{definition}
	\label{def:D}
	Let the space $\D$ consist of pairs $(u,F)$ such that
		\begin{align*}
		u &\in L^\infty(\R),\\
		u_x &\in L^2(\R),\\
		F &\in L^\infty(\R),\\
		F &\text{ is monotonically increasing},\\
		F &\text{ is left continuous},\\
		\lim_{x\rightarrow-\infty}F(x) &= 0,\\
		\|F\|_\infty &= F_\infty,\\
		\int_a^b u_x^2(x)\:\d x &\leq F(b^-)-F(a^+).
		\end{align*}
\end{definition} 

\begin{remark}
Given a pair $(u,F)\in \D$, there exists a positive finite Radon measure $\mu$, such that $F(x)=\mu((-\infty,x))$.
\end{remark}

Let $T_t$ be the conservative solution operator associated to \eqref{eq:HSsys}, as defined in \cite{N}, mapping every initial data $(u,F)$ to the corresponding solution at time $t$. For continuous and piecewise linear initial data $(u,F)$, the conservative solution of \eqref{eq:HSsys} takes a particularly simple form as long as no wave breaking takes place: The solution is again continuous and piecewise linear and the breakpoints $x_j(t)$ travel along characteristics,  i.e.\ along the curves $x_j(t)$ given by 
\begin{equation}\label{eq:characteristics2}
x_j(t) =  x_j(0)+u(0,x_j(0))t + \frac14\left(F(0,x_j(0))-\frac 12F_\infty\right)t^2,
\end{equation}
we get 
\begin{subequations}
	\label{eq:characteristics}
	\begin{align}
	u(t,x_j(t)) &= u(0,x_j(0)) + \frac12\left(F(0,x_j(0))-\frac 12F_\infty\right)t,\\
	F(t,x_j(t)) &= F(0,x_j(0)),
	\end{align} 
\end{subequations}
with linear interpolation between the breakpoints. Thus the equations \eqref{eq:characteristics} implicitly define the solution operator $T_t$ in the case of continuous and piecewise linear initial data $(u,F)$. 

Turning our attention once more towards Example~\ref{ex:peakon}, we see that the two curves 
\begin{equation*}
x_1(t)=-1+t-\frac14 t^2 \quad \text{ and }\quad x_2(t)=1-t+\frac14 t^2
\end{equation*}
describe the position of the breakpoints. Furthermore, at the breaking time $t=2$ we have $x_1(t)=x_2(t)$. In the general case of a continuous and piecewise linear initial data $(u,F)$, wave breaking occurs at times $t$ where at least two break points coincide, i.e., $x_j(t)=x_k(t)$ for some $j\not =k$.

Using the above observations, we will now derive the numerical scheme. The idea is to use piecewise linear projection operators $\p_{\Delta x}$ to project the solution at each time step, and $T_{\Delta t}$ to evolve the solution one time step $\Dt$ ahead. To improve the readability, we define points in space and time
	\begin{align*}
		t^n &= n\Delta t, \quad n\in\N,\\
		x_j &= j\Delta x, \quad j\in\Z.
	\end{align*}

\begin{definition}
	\label{def:P}
	Define the projection operator $\p_{\Delta x}:\D\rightarrow\D$ so that $(\bar u,\bar F) = \p_{\Delta x}(u,F)$ is given by
	\begin{align*}
		\bar u(x_j) &= u(x_j),\\
		\bar F(x_j) &= F(x_j),
	\end{align*}
	with linear interpolation in between grid points $\Delta x\Z$.
\end{definition}

\begin{remark}
\label{remark:P}
	The operator $\p_{\Delta x}$ is well defined since it is assumed that $F$ is (left) continuous, and thus one can evaluate $F$ at any point. 
\end{remark}

Assume now that the time step $\Dt$ is so small that no wave breaking occurs as the piecewise linear approximation is evolved from one time step to the next. Then the scheme is defined by $(U^0,F^0) = \p_{\Delta x}(u_0,F_0)$ and 
\begin{equation*}
	(U^{n+1},F^{n+1}) = \p_{\Delta x}T_{\Delta t}(U^n,F^n) \quad \text{ for } n\geq 0.
\end{equation*} 
We will need to interpret the numerical solution as a function from $[0,\infty)\times \R$ to $\R\times \R_+$.
\begin{definition}
	\label{def:Godunov num sol}
	We define the numerical solution $(u_{\Delta x},F_{\Delta x})$ at a point $(t,x)\in[0,T]\times\R$ by
	\begin{equation*}
		(u_{\Delta x},F_{\Delta x})(t,x) = \p_{\Delta x}T_\tau(U^n,F^n)(x)\quad \text{ for } t = \tau + t^n, \tau\in[0,\Delta t).
	\end{equation*}
	That is, we follow the solution along lines $x=x_j$ from one time step to the next, and interpolate linearly in between.
\end{definition}

After each evolution $\Dt$ forward in time, the solution is projected onto the space of continuous piecewise linear functions. As multiple peakons can be glued together to form multipeakons, which solve \eqref{eq:HSsys}, we can continue computing the solution forward in time after each projection.

\begin{remark}
Note that as the numerical approximation consists of linear interpolations between grid points and solving exactly between time steps, $F_\Dx$ satisfies $0 \leq F_\Dx(t,x) \leq F_\infty$.
\end{remark}

We introduce a CFL-like condition that ensures that characteristic curves $x_j(t)$ do not collide as long as we evolve the equations less than $\Delta t$. In particular, the condition prevents wave breaking, which occurs when $x_j(t) = x_{j+1}(t)$ for some $j\in\Z$ and $t>0$.  We arrive at the following bound on $\Delta t$ in terms of the initial data and the grid length $\Delta x$. The condition is not a true CFL condition in the sense that characteristics may travel past several cells $[x_j,x_{j+1}]$ during one time step.

\begin{definition}[CFL-like condition]
	\label{def:cfl}
	We require that $\Delta t$ satisfies
	\begin{equation}
		\label{eq:cfl}
		\Delta t \leq \frac{\alpha}{2\sqrt{F_\infty}}\sqrt{\Delta x}, \quad \alpha \in (0,1].
	\end{equation}
\end{definition}
Note that \eqref{eq:cfl} is less restrictive than the CFL conditions used for conservation laws, which reads $\Delta t < C\Delta x$ for some $C$ depending on the initial data and the particular flux function.
\begin{remark}\label{remark:wavebreaking}
In the upcoming proofs we will use
	\begin{equation}
		\label{eq:cfl2}
		\Delta t = \frac{1}{2\sqrt{F_\infty}}\sqrt{\Delta x}
	\end{equation}
	to prove convergence. From \eqref{eq:characteristics2} we find that if condition \eqref{eq:cfl2} holds, the characteristics $x_j(t)$ and $x_{j+1}(t)$ starting from neighbouring grid points are at least a distance $\frac12 \Delta x$ apart for all $0\leq t\leq \Delta t$, i.e.,  $x_j(t) + \frac12\Delta x < x_{j+1}(t)$ for all $t\in[0,\Delta t]$.
\end{remark}

\begin{remark}
Note that we could have chosen any fixed $\alpha\in (0,1]$ to take the step from \eqref{eq:cfl} to \eqref{eq:cfl2} (with 1 replaced by $\alpha$). As a consequence the least distance between characteristics $x_j(t)$ and $x_{j+1}(t)$, starting from neighboring grid points, would be given by $\beta(\alpha)\Dx$ and could be computed using \eqref{eq:characteristics2}.
\end{remark}

Similarly to the forward characteristics governed by \eqref{eq:characteristics2}, there are characteristics backwards in time. In particular, we can associate to any grid point $(x_j,\tau)$ with $t^n\leq \tau\leq t^{n+1}$, the unique point $(t^n, \xi_j^n(\tau))$  given by 
\begin{equation}\label{eq:backwards2} 
\xi_j^n(\tau) = x_j - u(t^n,\xi_j^n(\tau))(\tau-t^n)+\frac14\left(F(t^n,\xi_j^n(\tau)))-\frac12F_\infty\right)(\tau-t^n)^2
\end{equation}
and 
	\begin{align*}
		u\left(\tau, x_j\right) &= u(t^n,\xi_j^n(\tau)) -\frac12\left(F(t^n,\xi_j^n(\tau))-\frac12F_\infty\right)(\tau-t^n),\\
		F\left(\tau,x_j\right) &= F(t^n,\xi_j^n(\tau)).
	\end{align*}
\begin{remark}
The numerical scheme can be written in the more familiar form
\begin{align*}
	U_i^{n+1} &= U_j^n+\frac 12\left( F_j^n-\frac 12F_\infty\right)\Dt\\ &\quad - \frac{U_j^n+\frac 14(F_j^n-\frac 12F_\infty)\Dt}{1+\frac{U_{j+1}^n-U_j^n}{\Dx}\Dt+\frac{F_{j+1}^n-F_j^n}{\Dx}\Dt^2}\frac{\Dt}{\Dx}\left(U_{j+1}^n+\frac 12F_{j+1}^n\Dt - U_j^n-\frac 12F_j^n\Dt\right),\\
	F_i^{n+1} &=  F_j^n - \frac{U_j^n+\frac 14(F_j^n-\frac 12F_\infty)\Dt}{1+\frac{U_{j+1}^n-U_j^n}{\Dx}\Dt+\frac{F_{j+1}^n-F_j^n}{\Dx}\Dt^2}\frac{\Dt}{\Dx}\left(F_{j+1}^n- F_j^n\right),
\end{align*}
where the backward characteristic from $x_i$ at $t^{n+1}$ satisfies $\xi_i^n(\Dt)\in[x_j,x_{j+1}]$, see \eqref{eq:backwards2}. 
\end{remark}

\subsection{A priori bounds of the numerical solutions}\label{sec:aprioribounds}
In this section, we prove certain a priori bounds of the proposed method, which are needed to prove convergence. We begin with some preliminary results on the projection operator $\p_{\Dx}$.
\begin{proposition}
	\label{prop:proj est}
	For $(u,F)$ in $\D$, let $(u_p,F_p) = \p_{\Delta x}(u,F)$. Then we have the following estimates
	\begin{subequations}
	\begin{align*}
	\|u-u_p\|_\infty &\leq \sqrt{F_\infty}\sqrt{\Delta x},\\
	\|u-u_p\|_2 &\leq \sqrt{F_\infty}\Delta x,\\
	\|F-F_p\|_1 &\leq F_\infty\Delta x,\\
	\|F-F_p\|_2 &\leq F_\infty\sqrt{\Delta x}.
	\end{align*}
	\end{subequations}
\end{proposition}
\begin{proof}
	For any grid point $x_j$ we have $u(x_j)=u_p(x_j)$ and $F(x_j)=F_p(x_j)$ by the definition of $\p_{\Delta x}$. Hence, using the properties in Definition \ref{def:D}, for any $x\in [x_j,x_{j+1}]$ it holds that 
	\begin{align*}
	|u(x)-u_p(x)| &= \left|\frac{x_{j+1}-x}{\Delta x}(u(x)-u(x_j))+\frac{x-x_j}{\Delta x}(u(x)-u(x_{j+1})\right|\\ 
	& \leq \frac{x_{j+1}-x}{\Delta x}\sqrt{x-x_j}\sqrt{F(x)-F(x_j)} \\
	& \ +\frac{x-x_j}{\Delta x}\sqrt{x_{j+1}-x}\sqrt{F(x_{j+1})-F(x)}\\
	&\leq \sqrt{F_\infty}\sqrt{\Delta x},
	\end{align*}
	which proves the first inequality. Next, we have
	\begin{align*}
	\|u-u_p\|_2^2 &= \sum_{j\in\Z}\int_{x_j}^{x_{j+1}}\big(u(x)-u_{\Delta x}(x)\big)^2\:\d x\non
	&= \sum_{j\in\Z}\int_{x_j}^{x_{j+1}}\Bigg(\frac{x_{j+1}-x}{\Delta x}\big(u(x)-u(x_j)\big)+\frac{x-x_j}{\Delta x}\big(u(x)-u(x_{j+1})\big)\Bigg)^2\d x\non
	&\leq \sum_{j\in\Z}\int_{x_j}^{x_{j+1}}\Bigg(\frac{x_{j+1}-x}{\Delta x}\sqrt{x-x_j}\sqrt{F(x)-F(x_j)} \non
	& \qquad \qquad \qquad \qquad +\frac{x-x_j}{\Delta x}\sqrt{x_{j+1}-x}\sqrt{F(x_{j+1})-F(x)}\Bigg)^2\:\d x\non
	&\leq\sum_{j\in\Z}\int_{x_j}^{x_{j+1}}\big(F(x_{j+1})-F(x_j)\big)\Delta x\:\d x\non
	&\leq F_\infty \Delta x^2, \nonumber
	\end{align*}
	and thus $\|u-u_p\|_2\leq \sqrt{F_\infty}\Delta x$. The $L^1$-estimate for $F$ is proved as follows,
	\begin{align*}
	\|F-F_p\|_1 &= \sum_{j\in\Z}\int_{x_j}^{x_{j+1}}\left|F(x)-F_{\Delta x}(x)\right|\:\d x \non
	&\leq \sum_{j\in\Z}\int_{x_j}^{x_{j+1}}F(x_{j+1})-F(x
	_j)\:\d x \non
	&\leq \sum_{j\in\Z}\left(F(x_{j+1})-F(x
	_j)\right)\Delta x\non
	&\leq F_\infty\Delta x.
	\end{align*}
	From the $L^1$-estimate one can obtain the $L^2$-estimate,
	\begin{align*}
	\|F-F_p\|_2^2 &= \sum_{j\in\Z}\int_{x_j}^{x_{j+1}}\left|F(x)-F_{\Delta x}(x)\right|^2\:\d x \non
	&\leq \sum_{j\in\Z}\int_{x_j}^{x_{j+1}}\left(F(x_{j+1})-F(x_j)\right)^2\:\d x \non
	&\leq \sum_{j\in\Z}\left(F(x_{j+1})-F(x_j)\right)^2\Delta x \non
	&\leq F_\infty^2\Delta x.
	\end{align*}
\end{proof}

To prove that the numerical approximation converges, we wish to employ the Arzel{\`a}--Ascoli theorem to ensure convergence of a subsequence of $u_{\Delta x}$, and subsequently a version of the Kolmogorov compactness theorem to get convergence of a subsequence of $F_{\Delta x}$. To invoke the Arzel{\`a}--Ascoli theorem, we need $u_{\Delta x}$ to be uniformly equicontinuous and equibounded. For the Kolmogorov compactness theorem we need that $F_{\Delta x}$ is of uniformly bounded total variation, that $F_{\Delta x}(t,\cdot)$ is continuous in $t$ in the $L^1(\R)$-norm, and  that $F_{\Delta x}(t,\cdot)$ does not escape to infinity as $\Delta x$ tends to zero. First we establish some immediate properties of the solutions $(u_{\Delta x},F_{\Delta x})$.

\begin{lemma}
	\label{lemma:basic}
	The numerical solution $(u_{\Delta x},F_{\Delta x})$ satisfies
	\begin{subequations}
		\begin{align}
		|u_{\Delta x}(t,x)| &\leq \|u_0\|_{L^\infty(\R)}+\frac 14F_\infty t,\\
		0 &\leq F_{\Delta x}(t,x) \leq F_\infty \\ \label{eq:basic3}
		\int_a^b u_{\Delta x,x}^2(t,x)\:\d x &\leq F_{\Delta x}(t,b)-F_{\Delta x}(t,a), \quad \text{ for all } a\leq b.		
		\end{align}
\label{eq:basic}
	\end{subequations}
Moreover, $F_{\Delta x}(t,\cdot)$ is continuous and monotonically increasing. If $\mathrm{supp}\: \mu_0 \subseteq [a,b]$, then $\mathrm{supp}\: F_{\Delta x,x}(t,\cdot)\subseteq [a(t),b(t)]$ for some smooth curves $a(t),b(t)$. Finally, if $T.V.(u_0)<\infty$ we have the estimate $T.V.(u_{\Delta x}(t)) \leq T.V.(u_0)+\frac 12F_\infty t$.
\end{lemma}

\begin{proof}
The bounds on $u_{\Delta x}(t,x)$ and $F_{\Delta x}(t,x)$ follow from \eqref{eq:characteristics} and Definition \ref{def:Godunov num sol}. Since both \eqref{eq:characteristics} and the projection operator preserve the monotonicity of $F$, we have that $F_{\Delta x}$ is monotone increasing. Continuity follows from the fact that characteristics emanating from different grid points are at least $\frac 12\Delta x$ apart as long as the time step is controlled by \eqref{eq:cfl2}.

We show $\int_a^b u^2_{\Delta x,x}(t,x)\:\d x \leq F_{\Delta x}(t,b)-F_{\Delta x}(t,a)$ for all $a\leq b$. To begin with let $t=0$. Since $u_{\Delta x}(0,\cdot)$ and $F_{\Delta x}(0,\cdot)$ are both piecewise linear and continuous it suffices to show the result for $x_j\leq a\leq b\leq x_{j+1}$.  By assumption one has that $\int_a^b u_x^2(0,x)dx\leq F(0,b)-F(0,a)$ and direct calculations yield
\begin{align*}
\int_a^b u_{\Delta x,x}^2(0,x)\:\d x &= (b-a)\left(\frac{ u(0,x_{j+1})- u(0,x_j)}{\Delta x}\right)^2\non
	&\leq \frac{b-a}{\Delta x}\big(F(0,x_{j+1})- F(0,x_j)\big)\non
	&\leq F_{\Delta x}(0,b)-F_{\Delta x}(0,a).
\end{align*}

Now, let $t = t^n+\tau$, and denote by $\tau\mapsto (\tilde{u}(\tau),\tilde{F}(\tau))$ the conservative solution with initial data $(u_{\Delta x}(t^n),F_{\Delta x}(t^n))$. Furthermore, assume that  $(u_{\Delta x}(t^n),F_{\Delta x}(t^n))$ satisfies \eqref{eq:basic3}. Then we have for each spatial grid point $x_j$ that $\tilde u(\tau,x_j) = u_{\Delta x}(t^n+\tau,x_j)$ and $\tilde F(\tau,x_j) = F_{\Delta x}(t^n+\tau,x_j)$. Moreover
\begin{equation*}
\int_a^b \tilde{u}_x^2(\tau,x)\:\d x \leq \tilde F(\tau,b)-\tilde F(\tau,a),
\end{equation*}
since this property is preserved along characteristics. Applying the projection operator, we can follow the same lines as in the case $t=0$, to obtain that \eqref{eq:basic3} holds for all $t\in [t^n, t^{n+1}]$.

By assumption $\mathrm{supp}\: \mu_0 \subseteq [a,b]$.  Let $x_{j-}$ be the closest gridpoint to $a$ from below, and let $x_{j+}$ be the closest gridpoint to $b$ from above. Then $F_{\Delta x,x}(0,\cdot)$ is supported in $[x_{j-}, x_{j+}]\subseteq [a-\Delta x, b+\Delta x]$. Furthermore, $F_{\Delta x}(0,x_{j-}) = 0$, $F_{\Delta x}(0,x_{j+}) = F_\infty$, $u_{\Delta x}(0,x_{j-}) = u_{\mathrm{left}}$, and $u_{\Delta x}(0,x_{j+}) = u_{\mathrm{right}}$.  

Next we show that also $F_{\Delta x,x}(\Delta t,\cdot)$ is compactly supported. By \eqref{eq:characteristics2}, we have $x_{j-}(\Delta t) = x_{j-}+u_{\mathrm{left}}\Delta t - \frac 18F_\infty\Delta t^2$ and $x_{j+}(\Delta t) = x_{j+}+u_{\mathrm{right}}\Delta t + \frac 18F_\infty\Delta t^2$. Thus $F_{\Delta x,x}(\Delta t)$ is supported in the interval $[a+u_{\mathrm{left}}\Delta t -\frac 18F_\infty\Delta t^2-2\Delta x,b + u_{\mathrm{right}}\Delta t + \frac 18F_\infty\Delta t^2+2\Delta x]$. Iteratively, we get that $F_{\Delta x,x}(k\Delta t)$ is supported in
\begin{equation*}
	\Big[a + u_{\mathrm{left}}k\Delta t + \frac 18F_\infty(k\Delta t)^2-(k+1)\Delta x, b + u_{\mathrm{right}}k\Delta t + \frac 18F_\infty(k\Delta t)^2+(k+1)\Delta x \Big].
\end{equation*}
Here it is essential that 
\begin{equation*}
u_{\mathrm{left}}(k\Delta t)=u_{\mathrm{left}}-\frac14F_\infty k \Delta t.
\end{equation*}
Since $\Delta x = 4F_\infty \Delta t^2$, we have that $(k+1)\Delta x = \Delta x + 4F_\infty (k\Delta t)\Delta t$. From the interpolation between temporal grid points we get
\begin{align*}
	\label{eq:support}
	\mathrm{supp}\, F_{\Delta x,x}(t) & \subseteq [a(t), b(t)], \\
a(t) & = a + u_{\mathrm{left}}t -\Big(\frac 18 t+4\Dt\Big)F_\infty t - 2\Delta x, \non
b(t) & = b + u_{\mathrm{right}}t + \Big(\frac 18 t+4\Dt\Big)F_\infty t + 2\Delta x. \nonumber
\end{align*}
The total variation estimate follows from the fact that it holds for conservative solutions, and that the projection operator can only reduce the total variation. 
\end{proof}

\begin{remark}[Spatial H\" older continuity]\label{rem:x-holder}
An immediately derivable property of the numerical solution from \eqref{eq:basic3} is spatial H\" older continuity of $u_\Dx$:
\begin{align*}
|u_\Dx(t,x)-u_\Dx(t,y)| \leq \sqrt{F_\infty}\sqrt{|x-y|}.
\end{align*}
\end{remark}

In order to obtain temporal H\" older continuity for $u_{\Delta x}$ we will need to compare a numerical solution with itself several time steps ahead.
\begin{lemma}
	\label{lemma:DoD}
	For each $i,n,k$ there are non-negative constants $\beta_j^{ink}$ such that
	\begin{subequations}
	\begin{align}
	F^{n+k}_i &= \sum_{j=-kC_{\Delta x}}^{kC_{\Delta x}} \beta_j^{ink}F_{i+j}^n,\\
	U^{n+k}_i &= \sum_{j=-kC_{\Delta x}}^{kC_{\Delta x}} \beta_j^{ink}\left(U_{i+j}^n+\frac12F_{i+j}^nk\Delta t\right)-\frac 14F_\infty k\Delta t,\\
	\sum_{j=-kC_{\Delta x}}^{kC_{\Delta x}}\beta_j^{ink} &= 1,\label{eq:betasum}
	\end{align}
	\end{subequations}
	where
	\begin{equation*}
	C_{\Delta x}=\left\lceil\left(\|u_0\|_\infty+\frac14 F_\infty t^{n+k}\right)\frac{\Delta t}{\Delta x}\right\rceil.
	\end{equation*}
\end{lemma}
\begin{proof}
We prove the lemma by induction on $k$. First note that the statement is trivially true for $k=0$. Then assume that it holds for $k=l$. We show that it must then hold for $k=l+1$ as well. We have that 
\begin{equation*}
|\xi_i^{n+l}(\Delta t)-x_i| \leq \sup_i\vert U_i^{n+l+1}\vert\Delta t  +\frac14 F_\infty\Delta t^2\leq \left(\|u_0\|_\infty+\frac14 F_\infty t^{n+l+1}\right)\Delta t,
\end{equation*}
where $\xi_i^{n+l}(\Delta t)$ is a backwards characteristic, cf. \eqref{eq:backwards2}. Hence, if we define $\tilde C_{\Delta x}=\left\lceil\left(\|u_0\|_\infty+\frac14 F_\infty t^{n+l+1})\right)\frac{\Delta t}{\Delta x}\right\rceil$ we have that $x_j\leq \xi_i^{n+l}(\Delta t)\leq x_{j+1}$ for some $j$ such that $|i-j|\leq \tilde C_{\Delta x}$ and $|i-j-1|\leq \tilde C_{\Delta x}$. Furthermore, we have
\begin{equation*}
	F_i^{n+l+1} = \frac{\xi_i^{n+l}(\Delta t)-x_j}{\Delta x}F_{j+1}^{n+l} + \frac{x_{j+1}-\xi_i^{n+l}(\Delta t)}{\Delta x}F_j^{n+l}.
\end{equation*}
Let \begin{align*}
s = \frac{\xi_i^{n+l}(\Delta t)-x_j}{\Delta x}.
\end{align*}
Since $\tilde C_{\Delta x}$ is greater than the $C_{\Delta x}$ in the inductive assumption, we get
\begin{align*}
	F_i^{n+l+1} &=s\sum_{j' = -lC_{\Delta x}}^{lC_{\Delta x}}\beta_{j'}^{(j+1)nl}F_{j+1+j'}^n + (1-s)\sum_{j' = -lC_{\Delta x}}^{lC_{\Delta x}}\beta_{j'}^{jnl}F_{j+j'}^n \non
	&= \sum_{j = -(l+1)\tilde C_{\Delta x}}^{(l+1)\tilde C_{\Delta x}}\beta_j^{in(l+1)}F_{i+j}^n,
\end{align*}
with
\begin{align*}
\sum_{j = -(l+1)\tilde C_{\Delta x}}^{(l+1)\tilde C_{\Delta x}}\beta_j^{in(l+1)} = \sum_{j' = -lC_{\Delta x}}^{lC_{\Delta x}}s\beta_{j'}^{(j+1)nl} + \sum_{j' = -lC_{\Delta x}}^{lC_{\Delta x}}(1-s)\beta_{j'}^{jnl} = 1.
\end{align*}

The computation for $U_i^{n+k}$ is analogous. Indeed, we have 
\begin{align*}
U_i^{n+l+1} 
&= \frac{\xi_i^{n+l}(\Delta t)-x_j}{\Delta x}U_{j+1}^{n+l} + \frac{x_{j+1}-\xi_i^{n+l}(\Delta t)}{\Delta x}U_j^{n+l} \\
& \quad +\frac12\Bigg(\frac{\xi_i^{n+l}(\Delta t)-x_j}{\Delta x}F_{j+1}^{n+l} + \frac{x_{j+1}-\xi_i^{n+l}(\Delta t)}{\Delta x}F_j^{n+l}\Bigg)\Delta t-\frac14 F_\infty \Delta t \\
& = \frac{\xi_i^{n+l}(\Delta t)-x_j}{\Delta x}\sum_{j' = -lC_{\Delta x}}^{lC_{\Delta x}}\beta_{j'}^{(j+1)nl}\big(U_{j+1+j'}^n+\frac12F_{j+1+j'}^nl\Delta t\big)\\
& \quad + \frac{x_{j+1}-\xi_i^{n+l}(\Delta t)}{\Delta x}\sum_{j' = -lC_{\Delta x}}^{lC_{\Delta x}}\beta_{j'}^{jnl}\big(U_{j+j'}^n+\frac12F_{j+j'}^nl\Delta t\big)\\
& \quad + \frac12 \frac{\xi_i^{n+l}(\Delta t)-x_j}{\Delta x}\sum_{j' = -lC_{\Delta x}}^{lC_{\Delta x}}\beta_{j'}^{(j+1)nl}F_{j+1+j'}^n \\ & \quad + \frac12\frac{x_{j+1}-\xi_i^{n+l}(\Delta t)}{\Delta x}\sum_{j' = -lC_{\Delta x}}^{lC_{\Delta x}}\beta_{j'}^{jnl}F_{j+j'}^n
 -\frac14 F_\infty (l+1)\Delta t \\
& =  \sum_{j = -(l+1)\tilde C_{\Delta x}}^{(l+1)\tilde C_{\Delta x}}\beta_j^{in(l+1)}\Big(U_{i+j}^n+\frac12 F_{i+j}^n(l+1)\Delta t\Big)- \frac14 F_\infty (l+1)\Delta t.
\end{align*}
\end{proof}

Next is an important corollary which provides a discrete H\" older continuity estimate for the numerical solution $u_{\Delta x}$.
\begin{corollary}[Discrete temporal H\"older continuity]
\label{corollary:disc holder}
The numerical solution satisfies
\begin{equation*}
|U_i^{n+k}-U_i^n| \leq C\sqrt{k\Delta t},
\end{equation*}
with
\begin{equation*}
		C=\sqrt{F_\infty}\sqrt{\Bigg(\|u_0\|_\infty+\frac14F_\infty t^{n+k}\Bigg)+2\sqrt{F_\infty}\sqrt{\Delta x}}+\frac14 F_\infty \sqrt{t^{n+k}}.
\end{equation*}
\end{corollary}

\begin{proof}
Using Lemma \ref{lemma:DoD}, we compute
\begin{align*}
U_i^{n+k}-U_i^n &=  \sum_{j=-kC_{\Delta x}}^{kC_{\Delta x}} \beta_j^{ink}\left(U_{i+j}^n+\frac12F_{i+j}^nk\Delta t\right)-\frac 14F_\infty k\Delta t - U^n_i\non
	&= \sum_{j=-kC_{\Delta x}}^{kC_{\Delta x}} \beta_j^{ink}\left(U_{i+j}^n-U_i^n\right) + \sum_{j=-kC_{\Delta x}}^{kC_{\Delta x}} \beta_j^{ink}\left(\frac12F_{i+j}^nk\Delta t\right)-\frac 14F_\infty k\Delta t,
\end{align*}
and thus, remembering Remark \ref{rem:x-holder}, \eqref{eq:betasum}, and \eqref{eq:cfl2},
\begin{align*}
\big|U_i^{n+k}&-U_i^n\big| \\
&\leq \sum_{j=-kC_{\Delta x}}^{kC_{\Delta x}} \beta_j^{ink}\left|U_{i+j}^n-U_i^n\right| + \frac 14F_\infty k\Delta t\non
	&\leq \sqrt{F_\infty}\sqrt{kC_{\Delta x}\Delta x} + \frac 14F_\infty k\Delta t\non
	&\leq \sqrt{F_\infty}\sqrt{k\left(\|u_0\|_\infty+\frac14 F_\infty t^{n+k}\right)\Delta t+k\Delta x} + \frac 14F_\infty k\Delta t\non
	&\leq \sqrt{F_\infty}\sqrt{k\left(\|u_0\|_\infty+\frac14 F_\infty  t^{n+k}\right)\Delta t+2k\sqrt{\Delta x}\sqrt{F_\infty}\Delta t} + \frac 14F_\infty k\Delta t\non
	&\leq \left(\sqrt{F_\infty}\sqrt{\Bigg(\|u_0\|_\infty+\frac14F_\infty t^{n+k}\Bigg)+2\sqrt{F_\infty}\sqrt{\Delta x}}+\frac14 F_\infty \sqrt{t^{n+k}}\right) \sqrt{k\Delta t}.
\end{align*}
\end{proof}

We are now ready to prove that for each $T>0$ the solutions $u_{\Delta x}$ are uniformly H\" older continuous on $[0,T]\times\R$. Uniform H\" older continuity implies equicontinuity, which is necessary for the Arzel{\`a}--Ascoli theorem.

\begin{lemma}[H\" older continuity]
	\label{lemma:holder}
	Let $0\leq t,s\leq T$ and $x,y\in\R$, then
	\begin{equation*}
	|u_{\Delta x}(t,x)-u_{\Delta x}(s,y)| \leq C\sqrt{|t-s|+|x-y|},
	\end{equation*}
	where
	\begin{align*}
	C&  = 4\max\Bigg\{4\sqrt{F_\infty}\sqrt{\|u_0\|_\infty+\frac14F_\infty T}, 2\sqrt{F_\infty}, \\
	& \qquad \qquad \sqrt{F_\infty}\sqrt{(\|u_0\|_\infty+\frac14F_\infty T)+2\sqrt{F_\infty}\sqrt{\Delta x}}+\frac14 F_\infty \sqrt{T}\Bigg\}.
	\end{align*}
\end{lemma}
\begin{proof}
Assume first that $t^n\leq s<t\leq t^{n+1}$ and $x_j\leq x\leq x_{j+1}$. We start by adding and subtracting $u_{\Delta x}(s,x)$ and obtain  
\begin{equation*}
u_{\Delta x}(t,x) - u_{\Delta x}(s,y) = u_{\Delta x}(t,x) - u_{\Delta x}(s,x) + u_{\Delta x}(s,x) - u_{\Delta x}(s,y).
\end{equation*}
Then, we have by definition,
\begin{align*}
	u_{\Delta x}(t,x)-u_{\Delta x}(s,x) & = \frac{x-x_j}{\Delta x}\big(u_{\Delta x}(t,x_{j+1})-u_{\Delta x}(s,x_{j+1})\big) \\
& \qquad	+ \frac{x_{j+1}-x}{\Delta x}\big(u_{\Delta x}(t,x_j)-u_{\Delta x}(s,x_j)\big).
\end{align*}
Note that at the spatial grid points $x_l$ the solution $u_{\Delta x}(t,x_l)$ equals the conservative solution given by \eqref{eq:characteristics} with initial data $\left(u_{\Delta x}(t^n,\cdot),F_{\Delta x}(t^n,\cdot)\right)$ evolved $t-t^n<\Delta t$ forward in time. For conservative solutions given by \eqref{eq:characteristics} we do have H\" older continuity with the constant $C$ depending on $F_\infty$, $\|u_0\|_\infty$, and $T$ only. To be more specific it has been shown in the proof of \cite[Theorem 3.14]{GN} that
\begin{equation*}
\vert u_{\Dx}(t,x_j)-u_{\Dx}(s,x_j)\vert \leq \sqrt{F_\infty}\sqrt{\|u_0\|_\infty+\frac14F_\infty t}\sqrt{|t-s|}+\frac14 F_\infty|t-s|,
\end{equation*}
for all $j\in \mathbb{Z}$.
Hence, we have
\begin{align*}
	\big|u_{\Delta x}& (t,x) - u_{\Delta x}(s,y)\big| \\
	&= \big|u_{\Delta x}(t,x) - u_{\Delta x}(s,x) + u_{\Delta x}(s,x) - u_{\Delta x}(s,y)\big|\non
	&\leq \big|u_{\Delta x}(t,x) - u_{\Delta x}(s,x)\big| +\big| u_{\Delta x}(s,x) - u_{\Delta x}(s,y)\big|\non
	&\leq \frac{x-x_j}{\Delta x}\big|u_{\Delta x}(t,x_{j+1})-u_{\Delta x}(s,x_{j+1})\big| + \frac{x_{j+1}-x}{\Delta x}\big|u_{\Delta x}(t,x_j)-u_{\Delta x}(s,x_j)\big|\non 
	&\quad+\big| u_{\Delta x}(s,x) - u_{\Delta x}(s,y)\big|\non
	&\leq \sqrt{F_\infty}\sqrt{\|u_0\|_\infty+\frac14F_\infty t}\sqrt{|t-s|}+\frac14 F_\infty|t-s| + \sqrt{F_\infty}\sqrt{|x-y|}\non
	&\leq K\sqrt{|t-s|+|x-y|},
\end{align*}
for $K = 2\max\big\{2\sqrt{F_\infty}\sqrt{\|u_0\|_\infty+\frac14F_\infty t},\sqrt{F_\infty}\big\}$.

We look at the general case $t^{n-1}\leq s\leq t^n\leq t^{n+k}\leq t\leq t^{n+k+1}$, and $x_{j-1}\leq y \leq x_j \leq x_{j+l}\leq x\leq x_{j+l+1}$. Then, by Corollary \ref{corollary:disc holder}, we have,
\begin{align*}
\big|u_{\Delta x}&(t,x) - u_{\Delta x}(s,y)\big| \\
&\leq \left|u_{\Delta x}(t,x) - U^{n+k}_{j+l}\right| + \left|U^{n+k}_{j+l}-U^n_j\right|+ \left|U^n_j - u_{\Delta x}(s,y)\right|\non
	&\leq K\sqrt{|t-t^{n+k}|+|x-x_{j+l}|} + \left|U^{n+k}_{j+l}-U^n_{j+l}\right|+ \left|U^n_{j+l}-U^n_j\right| \non &\quad + K\sqrt{|t^n-s|+|x_j-y|}\non
	&\leq K\sqrt{|t-t^{n+k}|+|x-x_{j+l}|}  \non
	& \quad + \Bigg(\sqrt{F_\infty}\sqrt{(\|u_0\|_\infty+\frac14F_\infty t^{n+k})+2\sqrt{F_\infty}\sqrt{\Delta x}}+\frac14 F_\infty \sqrt{t^{n+k}}\Bigg)\sqrt{k\Delta t}\non &\quad+\sqrt{F_\infty}\sqrt{l\Delta x}  + K\sqrt{|t^n-s|+|x_j-y|}\non
	&\leq 4\max\Bigg\{K,\sqrt{F_\infty}\sqrt{(\|u_0\|_\infty+\frac14F_\infty t^{n+k})+2\sqrt{F_\infty}\sqrt{\Delta x}}+\frac14 F_\infty \sqrt{t^{n+k}}\Bigg\} \non
	& \qquad \times\sqrt{|t-s|+|x-y|}.
\end{align*}
\end{proof}

To use the Kolmogorov compactness theorem we need uniform regularity of $F_{\Delta x}$ in $t$.

\begin{lemma}
	\label{lemma:F cont t}
	Let $0\leq t,s\leq T$, then
	\begin{equation*}
	\|F_{\Delta x}(t)-F_{\Delta x}(s)\|_1 \leq C\vert s-t\vert+D\Delta t+12 F_\infty\Delta x,
	\end{equation*}
	where  $C=6\big( \|u_0\|_{\infty}+\frac14 F_\infty (17\Delta t+T)\big)F_\infty$ and $D=8\big( \|u_0\|_{\infty}+\frac14 F_\infty (\Delta t+T)\big)F_\infty$.
\end{lemma}

\begin{proof}
To begin with let $t^n\leq s<t\leq t^{n+1}$. Assume first that there exists $j'$ such that $x_{j'}\leq\xi_j^n(t)<\xi_j^n(s)\leq x_{j'+1}$. Then
\begin{align*}
	F_{\Delta x}(t,x_j) - F_{\Delta x}(s,x_j) &= \frac{x_{j'+1}-\xi_j^n(t)}{\Delta x}F_{\Delta x}(t^n,x_{j'}) + \frac{\xi_j^n(t) -x_{j'}}{\Delta x}F_{\Delta x}(t^n,x_{j'+1}) \non &\quad - \frac{x_{j'+1}-\xi_j^n(s)}{\Delta x}F_{\Delta x}(t^n,x_{j'}) - \frac{\xi_j^n(s)-x_{j'}}{\Delta x}F_{\Delta x}(t^n,x_{j'+1})\non
	&= \frac{\xi_j^n(t)-\xi_j^n(s)}{\Delta x}\left(F_{\Delta x}(t^n,x_{j'+1})-F_{\Delta x}(t^n,x_{j'})\right).
\end{align*}
Otherwise, we have that there exist $j^-$ and $j^+$, possibly equal, such that $x_{j^--1}<\xi_j^n(t)\leq x_{j^-}\leq x_{j^+}\leq\xi_j^n(s)< x_{j^++1}$. Then
\begin{align*}
	F_{\Delta x}(s,x_j) - F_{\Delta x}(t,x_j) &= F_{\Delta x}(t^n,\xi_j^n(s))-F_{\Delta x}(t^n,\xi_j^n(t))\non
	&= \frac{x_{j-}-\xi_j^n(t)}{\Delta x}\left(F_{\Delta x}(t^n,x_{j^-})-F_{\Delta x}(t^n,x_{j^- -1})\right) \non & \quad + \frac{\xi_j^n(s)-x_{j^+}}{\Delta x}\left(F_{\Delta x}(t^n,x_{j^++1})-F_{\Delta x}(t^n,x_{j^+})\right)\non&\quad + \sum_{i=j^-}^{j^+-1}\left(F_{\Delta x}(t^n,x_{i+1})-F_{\Delta x}(t^n,x_i)\right).
\end{align*}
The number of terms in the above sum is bounded from above by $\vert \xi_j^n(t)-\xi_j^n(s)\vert$. Direct calculations yield
\begin{align*}
\big\vert \xi_j^n(t) & - \xi_j^n(s)\big\vert \\
& \leq \big\vert u_{\Delta_x}(t^n,\xi_j^n(s))(s-t^n)-u_{\Delta x}(t^n, \xi_j^n(t))(t-t^n)\big\vert \\
& \quad + \frac14 \Big\vert \Big(F_{\Delta x}(t^n, \xi_j^n(t))-\frac12 F_\infty\Big)(t-t^n)^2-\Big(F_{\Delta_x}(t^n, \xi_j^n(s)-\frac12F_\infty\Big)(s-t^n)^2\Big\vert \\
& \leq (\vert t-s\vert + 2\Delta t) \|u_{\Delta x}(t^n, \cdot)\|_{L^\infty} + \frac 14 F_\infty (\vert t-s\vert  +\Delta t) \Delta t\\
& \leq \big( \|u_0\|_{L^\infty}+\frac14 F_\infty (\Delta t+T)\big)(\vert t-s\vert +2\Delta t),
\end{align*} 
and therefore 
\begin{equation*}
\vert j^+-j^-\vert \leq \bigg\lfloor \Big( \|u_0\|_{\infty}+\frac14 F_\infty (\Delta t+T)\Big)\frac{\vert s-t\vert +2\Delta t}{\Delta x}\bigg\rfloor.
\end{equation*}
Due to condition \eqref{eq:cfl2} characteristics (forward as well as backward) from neighbouring grid points have a minimum distance of $\frac12 \Delta x$. Hence for each $j'$, the maximal number of backward characteristics $\xi_j^n(\Delta t)$ ending up in $[x_{j'}, x_{j'+1}]$ equals two.
Hence we have the bound
\begin{align*}
\int_\R\big|F_{\Delta x}(t,x)-F_{\Delta x}&(s,x)\big|\:\d x \\
&= \sum_{j\in\Z}\left|F_{\Delta x}(t,x_j)-F_{\Delta x}(s,x_j)\right|\Delta x\non
	&= \sum_{j\in\Z}\big|F_{\Delta x}(t^n,\xi_j^n(t))-F_{\Delta x}(t^n,\xi_j^n(s))\big|\Delta x\non
	&\leq 2\left\lfloor( \|u_0\|_{\infty}+\frac14 F_\infty (\Delta t+T))\frac{\vert s-t\vert +2\Delta t}{\Delta x}\right\rfloor\\
	& \qquad \qquad \times\sum_{j\in\Z}\big(F_{\Delta x}(t^n,x_{j+1})-F_{\Delta x}(t^n,x_j)\big)\Delta x\non&\quad + 6\sum_{j\in \Z}\left(F_{\Delta x}(t^n,x_{j+1})-F_{\Delta x}(t^n,x_j)\right)\Delta x\non
	&\leq 2\Big( \|u_0\|_{\infty}+\frac14 F_\infty (\Delta t+T)\Big)F_\infty(\vert s-t\vert +2\Delta t)+ 6F_\infty \Delta x.
\end{align*}
The general case $0\leq s<t\leq T$ can now be found by iteration over time steps and using condition \eqref{eq:cfl2},
\begin{align*}
	\|F_{\Delta x}(t)-F_{\Delta x}(s)\|_1& \leq 6\big( \|u_0\|_{\infty}+\frac14 F_\infty (17\Delta t+T)\big)F_\infty\vert s-t\vert\\
	& \quad +8\big( \|u_0\|_{\infty}+\frac14 F_\infty (\Delta t+T)\big)F_\infty\Delta t+12 F_\infty\Delta x.
\end{align*}
\end{proof}

\subsection{Convergence of the numerical solutions}\label{sec:conv}
In this section we prove that there exists a convergent subsequence of $(u_\Dx,F_\Dx)$, and that the limit is a conservative weak solution of \eqref{eq:HSsys}, which satisfies condition \eqref{cond:3}. First we rigorously define, as in \cite{BHR,GN,N}, conservative weak solutions.
\begin{definition}
	\label{def:weak}
	A pair $(u,F)$ is a conservative solution of \eqref{eq:HSsys} with initial data $(u_0,F_0)\in\D$ if 
	\begin{align*}
		u|_{t=0} = u_0\quad  &\text{ and } \quad F|_{t=0}=F_0\\
		u &\in C^{0,\frac 12}\left([0,T]\times\R\right), \text{ for all } T\geq 0,\\
		(u(t),F(t))&\in\D\text{ for all } t\geq 0,\\
		\|F(t)\|_\infty &= F_\infty \text{ for all } t\geq 0,
	\end{align*}
	and for all test functions $\phi\in C_c^\infty([0,\infty)\times\R)$ we have
	\begin{align} \nonumber
		\int_0^\infty\int_\R \phi_t(t,x)u(t,x)  + \phi_x(t,x)\frac12u(t,x)^2 + \phi(t,x)\left(\frac 12F(t,x)-\frac 14F_\infty\right)\:\d x\d t&  \\ 
		\quad + \int_\R\phi_0(x)u_0(x)\:\d x& = 0, \label{eq:weaku}\\ 
		\int_0^\infty\int_\R \phi_t(t,x)+u(t,x)\phi_x(t,x) \d\mu(t)\d t  + \int_\R\phi_0(x)\d\mu_0& = 0, \label{eq:weakF}
	\end{align}
	where $\mu(t)$ is the finite positive Radon measure with $F(t,\cdot)$ as its distribution function, see Definition \ref{def:D}.
\end{definition}
We prove the existence of a convergent subsequence of $(u_\Dx,F_\Dx)$.
\begin{theorem}
	\label{thm:conv}
	To any initial data $(u_0,F_0)\in\D$ such that $\mu_{0}$ has compact support, there exists a convergent subsequence of $\left(u_{\Delta x},F_{\Delta x}\right)$. The convergence is in $C\left([0,T],L^1(\R)\right)$, pointwise a.e. in $x$ for $F_{\Delta x}$, and uniform on $[0,T]\times\R$ for $u_{\Delta x}$. Moreover, the limit $(u,F)$ satisfies
	\begin{subequations}
	\begin{align*}
	u|_{t=0} = u_0\quad  &\text{ and } \quad F|_{t=0}=F_0\\
	u &\in C^{0,\frac 12}\left([0,T]\times\R\right),\\
	(u(t),F(t))&\in\D\text{ for all } t\geq 0,\\
	\|F(t)\|_\infty &= F_\infty.
	\end{align*}
	\end{subequations}
	Here the relation between the positive Radon measure $\mu_0$ and $F_0$ is given by $F_0(x)=\mu_0((-\infty,x))$.
\end{theorem}

\begin{proof}
We have from Lemma \ref{lemma:basic} that the family $u_{\Delta x}$ is uniformly bounded on $[0,T]\times \R$ and that $u_{\Delta x,x}(t,\cdot)$ has compact support for all $t\in[0,T]$. Furthermore, by Lemma \ref{lemma:holder} $u_{\Delta x}$ is uniformly equicontinuous. Hence the conditions for the Arzel{\`a}--Ascoli theorem are satisfied and there exists a convergent subsequence $(u_{\Dx,i},F_{\Dx,i})$ of $(u_\Dx,F_\Dx)$ such that $u_{\Dx,i}$ converges to some $u \in L^\infty([0,T]\times\R)$ for each $T>0$. The limit of $u_{\Dx,i}$ is bounded and H\" older continuous with the same constants as the individual $u_{\Delta x,i}$.

Next we show that the limit $u$ satisfies $u_x(t,\cdot)\in L^2(\R)$ for all $t\in[0,T]$. 
By construction we have that $\|u_{\Dx,x}(t,\cdot)\|_{L^2}\leq \sqrt{F_\infty}$ for all $t\in[0,T]$. Thus there exists a subsequence $(u_{\Dx,i_k}(t,\cdot), F_{\Dx,i_k}(t,\cdot))$ of $(u_{\Dx,i}(t,\cdot), F_{\Dx,i}(t,\cdot))$, so that $ u_{\Dx,x,i_k}(t,\cdot)$ converges weakly to $v$ in $L^2(\R)$. 
Thus for any $\phi\in C_c^\infty(\R)$ we have
\begin{align*}
	\int_\R\big(v(x) - u_x(t,x)\big)\phi(x)\:\d x &= \lim_{\Delta x\rightarrow 0}\int_\R\big(u_{\Delta x,x,i_k}(t,x)-u_x(t,x)\big)\phi(x)\:\d x\non
	&= - \lim_{\Delta x\rightarrow 0}\int_\R\big(u_{\Delta x,i_k}(t,x)-u(t,x)\big)\phi_x(x)\:\d x\non
	&= 0,
\end{align*}
and $v(\cdot) = u_x(t,\cdot)$. Thus we have $u_{\Delta x,x,i_k}(t,\cdot)\rightharpoonup u_x(t,\cdot)$ in $L^2(\R)$. A closer look reveals that the above argument shows that every weakly convergent subsequence has the same limit and therefore $u_{\Delta x,x,i}(t,\cdot)\rightharpoonup u_x(t,\cdot)$ in $L^2(\R)$ for all $t\in [0,T]$. 

Combining Lemma \ref{lemma:basic} and \cite[Theorem 12]{HH}, we obtain, that for each $t\in[0,T]$, there exists a subsequence $(u_{\Dx,i_j}(t,\cdot), F_{\Dx,i_j}(t,\cdot))$ of $(u_{\Dx,i}(t,\cdot),F_{\Dx,i}(t,\cdot))$ such that $F_{\Dx,i_j}(t,\cdot)$ converges pointwise everywhere and in the $L^1$-norm to a function of bounded variation. Following the lines of the proof of \cite[Theorem A.11]{HR} and taking into account Lemma~\ref{lemma:F cont t}, it then follows that there exists a subsequence of $(u_{\Dx,i}, F_{\Dx,i})$, for which the $F_{\Dx}$ converge in $C\left([0,T],L^1(\R)\right)$. Furthermore, denoting the limit by $F$, we even have pointwise almost everywhere convergence of a further subsequence to $F$.
 
Last but not least, we have a look at the connection between $u_x$ and $F$. Denote by $(\tilde u_{\Dx}, \tilde F_{\Dx})$ the very last subsequence of $(u_{\Dx}, F_{\Dx})$, then we have that   
\begin{align*}
	\int_a^bu_x^2(t,x)\:\d x &\leq \liminf_{\Delta x\rightarrow 0}\int_a^b \tilde u_{\Delta x,x}^2(t,x)\:\d x\non
	&\leq \liminf_{\Delta x\rightarrow 0}\left(\tilde F_{\Delta x}(t,b)-\tilde F_{\Delta x}(t,a)\right).
\end{align*}
Since, we can find two sequences $a_j\downarrow a$ and $b_j\uparrow b$ such that $\lim_{\Delta x\rightarrow 0}F_{\Delta x}(t,a_j) = F(t,a_j)$ and $\lim_{\Delta x\rightarrow 0}F_{\Delta x}(t,b_j) = F(t,b_j)$, we end up with 
\begin{align*}
	\int_a^bu_x^2(t,x)\:\d x &= \lim_{j\rightarrow\infty}\int_{a_j}^{b_j}u_x^2(t,x)\:\d x\non
	&\leq \lim_{j\rightarrow\infty}\liminf_{\Delta x\rightarrow 0}\int_{a_j}^{b_j} \tilde u_{\Delta x,x}^2(t,x)\:\d x\non
	&\leq \lim_{j\rightarrow\infty}\lim_{\Delta x\rightarrow 0}\big(\tilde F_{\Delta x}(t,b_j)-\tilde F_{\Delta x}(t,a_j)\big)\non
	&= \lim_{j\rightarrow\infty}\left(F(t,b_j)-F(t,a_j)\right)\non
	&= F(t,b^-)-F(t,a^+).
\end{align*}
\end{proof}

We still need to prove that the limit of the convergent subsequence is a conservative weak solution in the sense of Definition \ref{def:weak}.
\begin{theorem}
	\label{thm:weak sol}
	The limit $(u,F)$ from Theorem \ref{thm:conv} is a conservative solution in the sense of Definition \ref{def:weak}.
\end{theorem}

\begin{proof}
It remains to show that the integrals \eqref{eq:weaku} and \eqref{eq:weakF} hold. We compute the integrals for $\left(u_{\Delta x},F_{\Delta x}\right)$ as follows
\begin{align*}
& \int_0^\infty\int_\R \phi_t(t,x)u_{\Delta x}(t,x) + \phi_x(t,x)\frac12u_{\Delta x}^2(t,x) + \phi(t,x)\left(\frac 12F_{\Delta x}(t,x)-\frac 14F_\infty\right)\:\d x\d t\non
&\quad = \sum_{n\in\N_0}\int_0^{\Delta t}\int_\R\phi_t(t^n+\tau,x)u_{\Delta x}(t^n+\tau,x) \non 
&\qquad \qquad \qquad \qquad+ \phi_x(t^n+\tau,x)\frac12u_{\Delta x}^2(t^n+\tau,x) \non
&\qquad\qquad\qquad \qquad+ \phi(t^n+\tau,x)\left(\frac 12F_{\Delta x}(t^n+\tau,x)-\frac 14F_\infty\right)\:\d x\d\tau\non
&=  \sum_{n\in\N_0}\int_0^{\Delta t}\int_\R\phi_t(t^n+\tau,x)\big(u_{\Delta x}(t^n+\tau,x)-\tilde u_n(\tau,x)\big)\non 
&\qquad \qquad \qquad \qquad+ \phi_x(t^n+\tau,x)\frac12\left(u_{\Delta x}^2(t^n+\tau,x)-\tilde u_n^2(\tau,x)\right) \non
&\qquad \qquad \qquad\qquad + \phi(t^n+\tau,x)\left(\frac 12F_{\Delta x}(t^n+\tau,x)-\frac 12\tilde F_n(\tau,x)\right)\:\d x\d\tau\non
&\qquad + \sum_{n\in\N_0}\int_0^{\Delta t}\int_\R\phi_t(t^n+\tau,x)\tilde u_n(\tau,x) + \phi_x(t^n+\tau,x)\frac12\tilde u_n^2(\tau,x) \non
&\qquad \qquad \qquad \qquad \qquad + \phi(t^n+\tau,x)\left(\frac 12\tilde F_n(\tau,x)-\frac 14F_\infty\right)\:\d x\d\tau \\
&= \sum_{n\in\N_0} \RN{1}_n + \sum_{n\in\N_0} \RN{2}_n.
\end{align*}
Here $(\tilde u_n(\tau,x), \tilde F_n(\tau,x))$ denotes the conservative solution given by \eqref{eq:characteristics2} and \eqref{eq:characteristics} with initial data $(u_{\Delta x}(t^n,x), F_{\Delta x}(t^n,x))$. Since the conservative solution is a weak solution we get
\begin{align*}
\RN{2}_n&=\int_0^{\Delta t}\int_\R\phi_t(t^n+\tau,x)\tilde u_n(\tau,x) \non
& \quad + \phi_x(t^n+\tau,x)\frac12\tilde u_n^2(\tau,x) + \phi(t^n+\tau,x)\left(\frac 12\tilde F_n(\tau,x)-\frac 14F_\infty\right)\:\d x\d\tau \non
&= \int_\R\tilde u_n(\Delta t,x)\phi(t^{n+1},x)\:\d x-\int_\R u_{\Delta x}(t^n,x)\phi(t^n,x)\:\d x.
\end{align*}
Recalling that $u_{\Delta x}(t^n+\tau) = \p_{\Delta x}\tilde u_n(\tau)$ yields 
\begin{align*}
\sum_{n\in \N_0} \RN{2}_n& 
=\sum_{n\in\N_0}\Big(\int_\R\tilde u_n(\Delta t,x)\phi(t^{n+1},x)\:\d x-\int_\R u_{\Delta x}(t^n,x)\phi(t^n,x)\:\d x\Big) \\
&= -\int_\R u_{0\Delta x}\phi_0(x)\:\d x \non
& \quad + \sum_{n=1}^\infty \int_\R\big(\tilde u_{n-1}(\Delta  t,x)-\p_{\Delta x}\tilde u_{n-1}(\Delta t,x)\big)\phi(t^n,x)\:\d x\\
& = -\int_\R u_{0\Delta x}\phi_0(x)\:\d x +\mathcal{O}(\sqrt{\Delta x})
\end{align*}
where we applied Proposition~\ref{prop:proj est} in the last step as follows
\begin{align*}
\bigg|\sum_{n=1}^\infty \int_\R\big(\tilde u_{n-1}(\Delta  t,x)-&\p_{\Delta x}\tilde u_{n-1}(\Delta t,x)\big)\phi(t^n,x)\:\d x\bigg| \\
 &\leq \sum_{n=1}^\infty\big\|\tilde u_{n-1}(\Delta  t,\cdot)-\p_{\Delta x}\tilde u_{n-1}(\Delta t,\cdot)\big\|_{2}\|\phi(t^n,\cdot)\|_{2}\non
	&\leq   \sum_{n=1}^\infty\|\phi(t^n,\cdot)\|_{2}\sqrt{F_\infty}\Delta x\non
	&\leq  \sum_{n=1}^\infty\|\phi(t^n,\cdot)\|_{2}2F_\infty\Delta t\sqrt{\Delta x}\non
	&\leq 2\sup_{t\geq 0}\|\phi(t,\cdot)\|_2 T_\phi F_\infty\sqrt{\Delta x}.
\end{align*}
Here $T_\phi=\inf\{t\geq 0\mid \|\phi(s,\cdot)\|_2+\|\phi_t(s,\cdot)\|_2=0 \text{  for all } s\geq t\}$. Note that $T_\phi$ is finite since $\phi$ has compact support.

We now turn our attention to the first sum. Recall that $u_{\Delta x}(t^n+\tau) = \p_{\Delta x}\tilde u_n(\tau)$, $F_{\Delta x}(t^n+\tau) = \p_{\Delta x}\tilde F_n(\tau)$, and keep Proposition~\ref{prop:proj est} and Lemma~\ref{lemma:basic} in mind. Direct calculations then yield
\begin{align*}
\vert \RN{1}_n\vert &=\bigg|\int_0^{\Delta t}\int_\R\phi_t(t^n+\tau,x)\big(u_{\Delta x}(t^n+\tau,x)-\tilde u_n(\tau,x)\big) \non
& \qquad \qquad \qquad + \phi_x(t^n+\tau,x)\frac12\left(u_{\Delta x}^2(t^n+\tau,x)-\tilde u_n^2(\tau,x)\right)\non 
&\qquad \qquad \qquad+ \phi(t^n+\tau,x)\left(\frac 12F_{\Delta x}(t^n+\tau,x)-\frac 12\tilde F_n(\tau,x)\right)\:\d x\d\tau\bigg|\non
&\leq\sup_{\tau\in [0,\Delta t]}\big(\|\phi_t(t^n+\tau,\cdot)\|_{1} + \|\phi_x(t^n+\tau,\cdot)\|_{1} \|\tilde u_n(\tau,\cdot)\|_{\infty}\big) \non
& \qquad \qquad \qquad \times \|u_{\Delta x}(t^n+\tau,\cdot)-\tilde u_n(\tau,\cdot)\|_{\infty} \Delta t \non
& \qquad + \sup_{\tau\in [0,\Delta t]}\|\phi(t^n+\tau)\|_{\infty}\frac 12\big\|F_{\Delta t}(t^n+\tau,\cdot)-\tilde F(\tau,\cdot)\big\|_{1}\Delta t\non
&\leq\sup_{\tau\in [0,\Delta t]}\left(\big (\|\phi_t(t^n+\tau,\cdot, \cdot)\|_{1} + \|\phi_x(t^n+\tau,\cdot)\|_{1}\big)\big( \| u_0(\cdot)\|_{\infty}+\frac14 F_\infty(t^n+\tau)\big)\right)\non
& \qquad \qquad \qquad \times \sqrt{F_\infty} \sqrt{\Delta x}\Delta t \non
& \qquad + \sup_{\tau\in [0,\Delta t]}\|\phi(t^n+\tau, \cdot)\|_{\infty}\frac 12F_\infty \Delta x\Delta t.
\end{align*}
Since $\phi$ has compact support we end up with 
\begin{equation*}
\sum_{n\in\N_0} \RN{1}_n =\mathcal{O}(\sqrt{\Delta x}).
\end{equation*}
 In particular, we have
\begin{align*}
& \int_0^\infty\int_\R \phi_t(t,x)u_{\Delta x}(t,x) + \phi_x(t,x)\frac12u_{\Delta x}(t,x)^2 \\
& \qquad \qquad  + \phi(t,x)\left(\frac 12F_{\Delta x}(t,x)-\frac 14F_\infty\right)\:\d x\d t + \int_\R u_{0\Delta x}\phi_0(x)\:\d x = \mathcal{O}(\sqrt{\Dx}).
\end{align*}
By letting $\Dx \to 0$, we end up with \eqref{eq:weaku} as the subsequence of $(u_\Dx, F_\Dx)$, constructed in Theorem~\ref{thm:conv}, converges to $(u,F)$ uniformly in $[0,T]\times\R$ for $u_\Dx$ and in $C\left([0,T],L^1(\R)\right)$ for $F_{\Delta x}$.

In a similar fashion we demonstrate that the second integral equation \eqref{eq:weakF} must be satisfied as $\Dx \to 0$ as well. We have
\begin{align}
&\int_0^\infty\int_{\R}\phi_t(t,x)+u_{\Delta x}(t,x)\phi_x(t,x)\:\d\mu_{\Delta x}(t)\d t\non
	&= \int_0^\infty\int_{\R}\big(\phi_t(t,x)+u_{\Delta x}(t,x)\phi_x(t,x)\big)F_{\Delta x,x}(t,x)\:\d x\d t\non
	&= \sum_{n\in \N_0}\int_0^{\Delta t}\int_{\R}\big(\phi_t(t^n+\tau,x)+u_{\Delta x}(t^n+\tau,x)\phi_x(t^n+\tau,x)\big)\non
& \qquad\qquad\qquad	\times \left(F_{\Delta x,x}(t^n+\tau,x)-\tilde F_{n,x}(\tau,x)\right)\:\d x\d\tau\non
	&\quad + \sum_{n\in \N_0}\int_0^{\Delta t}\int_{\R}\big(u_{\Delta x}(t^n+\tau,x)-\tilde u_n(\tau,x)\big)\phi_x(t^n+\tau,x)\tilde F_{n,x}(\tau,x)\:\d x\d\tau\non
	&\quad + \sum_{n\in \N_0}\int_0^{\Delta t}\int_\R\big(\phi_t(t^n+\tau,x)+ \tilde u_n(\tau,x)\phi_x(t^n+\tau,x)\big)\tilde F_{n,x}(\tau,x)\:\d x\d\tau \non
& = \sum_{n\in \N_0} \RN{1}_n + \sum_{n\in \N_0} \RN{2}_n + \sum_{n\in \N_0} \RN{3}_n. \nonumber
\end{align}
Since conservative solutions are weak solutions we get
\begin{align*}
\RN{3}_n = \int_0^{\Delta t}\int_\R& \left(\phi_t(t^n+\tau,x)+ \tilde u_n(\tau,x)\phi_x(t^n+\tau,x)\right)\tilde F_{n,x}(\tau,x)\:\d x\d\tau\\
	&  = \int_{\R} \phi(t^{n+1},x)\tilde F_{n,x}(\Delta t,x)\:\d x-\int_{\R}\phi(t^n,x) F_{\Delta x,x}(t^n,x) \:\d x.
\end{align*}
Summation over $n$ then yields
\begin{align*}
&\sum_{n\in \N_0} \int_\R\phi(t^{n+1},x)\tilde F_{n,x}(\Delta t,x)-\phi(t^n,x)F_{\Delta x,x}(t^n,x) \:\d x\non 
&\quad =- \int_{\R}\phi_0(x)F_{0\Delta x,x}(x)\:\d x + \sum_{n=1}^\infty\int_{\R}\left(\tilde F_{n-1,x}(\Delta t,x)-F_{\Delta x,x}(t^n,x)\right)\phi(t^n,x)\:\d x.
\end{align*}
From the estimates in Proposition \ref{prop:proj est} we get
\begin{align*}
	\sum_{n=1}^\infty \Big\vert \int_{\R}\Big(\tilde F_{n-1,x}(\Delta t,x)&-F_{\Delta x,x}(t^n,x)\Big)\phi(t^n,x)\:\d x \Big\vert \\
	&=\sum_{n=1}^\infty \Big\vert \int_{\R}\left(\tilde F_{n-1}(\Delta t,x)-\p_{\Delta x}\tilde F_{n-1}(\Delta t,x)\right)\phi_x(t^n,x)\:\d x \Big \vert \non
	 &\leq \frac{T_\phi}{\Delta t} \sup_{n\geq 0}\big\|\tilde F_{n-1}(\Delta t,\cdot)-\p_{\Delta x}\tilde F_{n-1}(\Delta t,\cdot)\big\|_{1}\|\phi_x(t^n,\cdot)\|_{\infty}\non
	&\leq F_\infty\sup_{t\geq 0} \|\phi_x(t,\cdot)\|_{\infty} \Delta x\frac{T_\phi}{\Delta t}\non
	&\leq 2\sqrt{F_\infty}F_\infty \sup_{t\geq 0} \|\phi_x(t,\cdot)\|_{\infty}T_\phi \sqrt{\Delta x},
\end{align*}
and hence 
\begin{align*}
\sum_{n\in\N_0} \RN{3}_n =- \int_{\R}\phi_0(x)F_{0\Delta x,x}(x)\:\d x + \mathcal{O}(\sqrt{\Delta x}).
\end{align*}

The second term can be estimated as follows
\begin{align*}
\Bigg|& \sum_{n\in \N_0} \RN{2}_n\Bigg| \\
& = \left|\sum_{n\in \N_0}\int_0^{\Delta t}\int_{\R}\big(u_{\Delta x}(t^n+\tau,x)-\tilde u_n(\tau,x)\big)\phi_x(t^n+\tau,x)\tilde F_{n,x}(\tau,x)\:\d x\d\tau\right|\non
	&\leq \sum_{n\in \N_0}\int_0^{\Delta t}\int_{\R}\big|\phi_x(t^n+\tau,x)\big| \tilde F_{n,x}(t^n+\tau,x)\:\d x\|u_{\Delta x}(t^n+\tau,\cdot)-\tilde u_n(\tau,\cdot)\|_{\infty} \d\tau\non
	&\leq \sqrt{F_\infty}F_\infty \sup_{t\geq 0}\|\phi_x(t,\cdot)\|_{x} T_\phi\sqrt{\Delta x}. \nonumber
\end{align*}
It remains to prove that the first term tends to zero as $\Delta x\rightarrow 0$. We compute
\begin{align}
\Bigg|\sum_{n\in \N_0} \RN{1}_n\Bigg| &=\Bigg|\sum_{n\in \N_0}\int_0^{\Delta t}\int_{\R}\big(\phi_t(t^n+\tau,x)+u_{\Delta x}(t^n+\tau,x)\phi_\infty(t^n+\tau,x)\big)\non 
& \qquad \qquad \qquad \qquad \qquad \times\left(F_{\Delta x,x}(t^n+\tau,x)-\tilde F_{n,x}(\tau,x)\right)\:\d x\d\tau\Bigg|\non
	&= \Bigg|\sum_{n\in \N_0}\int_0^{\Delta t}\int_{\R}\left(\phi_{tx}(t^n+\tau,x)+\big(u_{\Delta x}(t^n+\tau,x)\phi_x(t^n+\tau,x)\right)_x\big)\non 
& \qquad \qquad \qquad \qquad \qquad \qquad \quad \times \left(F_{\Delta x}(t^n+\tau,x)-\tilde F_{n}(\tau,x)\right)\:\d x\d\tau\Bigg|\non
	&\leq T_\phi \sup_{t\geq 0}\|\phi_{tx}(t,\cdot)\|_{x} F_\infty\Delta x \non
	& \quad + T_\phi \sup_{0\leq t\leq T_{\phi}}\|u_{\Delta x}(t,\cdot)\|_{\infty}\|\phi_{xx}(t,\cdot)\|_{\infty} \Delta x F_\infty \non
	& \quad + T_\phi \sup_{0\leq t\leq T_\phi}\|u_{\Delta x,x}(t,\cdot)\|_{2}\|\phi_x(t,\cdot)\|_{\infty} F_\infty\sqrt{\Delta x}\non
	&=\mathcal{O} (\sqrt{\Delta x}). \nonumber
\end{align}
Finally, we end up with
\begin{align}\label{eq:weak_fin}
\int_0^\infty\int_{\R}\phi_t(t,x)+u_{\Delta x}(t,x)\phi_x(t,x)\:\d\mu_{\Delta x}(t)\d t + \int_{\R}\phi_0(x)F_{0\Delta x,x}(x)\:\d x = \mathcal{O}(\sqrt{\Delta x}).
\end{align}
Since for every $t$ we have that $F_\Dx (t,\cdot)\to F(t,\cdot)$ almost everywhere, the convergence of \eqref{eq:weak_fin} to \eqref{eq:weakF} as $\Dx \to 0$ follows from the proof of Proposition 7.19 in \cite{folland} and the fact that $u_\Dx \to u$ in $C([0,T]\times \R)$.
\end{proof}

A satisfactory uniqueness theory for conservative weak solutions of the Hunter--Saxton equation would have ensured that all limits in Theorem \ref{thm:conv} are equal, and thus that the sequence as a whole converges. Unfortunately, uniqueness of conservative solutions is unknown at the present time. On the other hand it is known that if the initial data $(u_0,F_0)$ is Lipschitz continuous, also the solution $(u(t,\cdot), F(t,\cdot))$ will be Lipschitz continuous for all $t\in [0,-\frac2{\min(u_{0,x})})$, at least. In particular, as Example~\ref{ex:peakon} and \ref{cusp} indicate, when wave breaking occurs $u(t,\cdot)$ may be H{\"o}lder, but not Lipschitz continuous and $F$ may even be discontinuous. 

In the next theorem, we consider the special case of weak conservative solutions $(u,F)$ such that $(u(t,\cdot), F(t,\cdot))$ are Lipschitz continuous for all $t\in[0,T]$.

\begin{theorem}\label{thm:smooth}
Let $(u,F)\in [W^{1,\infty}([0,T]\times\R)]^2$ be a conservative solution in the sense of Definition~\ref{def:weak}. Then the conservative solution is unique and there exists a constant $C>0$, dependent on $T$, $F_\infty$, and $\sup_{t\in[0,T]}\big(\|u_x(t,\cdot)\|_{\infty}+\|F_x(t,\cdot)\|_{\infty}\big)$, such that 
\begin{equation}\label{est:smooth}
		\sup_{t\in[0,T]}\big(\|(u_\Dx-u)(t,\cdot)\|_{\infty} + \|(F_\Dx-F)(t,\cdot)\|_{\infty}\big) \leq C\left(\sqrt{\Dx} + \Dx\right).
	\end{equation}
\end{theorem}

\begin{proof}
For any conservative solution $(u,F)\in [W^{1,\infty}([0,T]\times \R)]^2$ with initial data $(u_0,F_0)$, the characteristic equation 
\begin{equation*}
\frac \d{\d t}x(t) = u(t,x(t)), \qquad x(0) = x_0,
\end{equation*}
is uniquely solvable. Furthermore, the classical method of characteristics implies that the solution is unique and given by 
\begin{align*}
u(t,x(t))&=u_0(x_0)+\frac 12(F_0(x_0)-\frac 12F_\infty)t,\\
F(t,x(t))&=F_0(x_0).
\end{align*}
Introduce $(\tilde u^0, \tilde F^0)(t)=T_t\p_\Dx (u_0,F_0)$ and recall that $x_j(t)$ denotes the characteristic starting at the grid point $x_j$, then 
\begin{equation*}
(\tilde u^0(t,x_j(t)),\tilde F^0(t,x_j(t)))=(u(t,x_j(t)), F(t,x_j(t)))\quad \text{ for all } j\in \mathbb{N}.
\end{equation*} 
Moreover, for all $t\in[0,T]$, 
\begin{equation*}
\|\tilde u_x^0(t,\cdot)\|_\infty\leq \|u_x(t,\cdot)\|_\infty \quad \text{ and }\quad \|\tilde F_x^0(t,\cdot)\|_\infty\leq \|F_x(t,\cdot)\|_\infty.
\end{equation*}
For $n\geq 1$, define $(\tilde u^n, \tilde F^n)$ by 
\begin{equation*}
	(\tilde u^n,\tilde F^n)(t) =
	\begin{cases}
		T_{(t-t^n)}\p_\Dx (\tilde u^{n-1},\tilde F^{n-1})(t^n), &\qquad t\geq t^n,\\
		(\tilde u^{n-1},\tilde F^{n-1})(t), &\qquad t<t^n.
	\end{cases}
	\end{equation*}
Then, following the same lines, one has 
\begin{equation*}
\| \tilde u^n_x(t,\cdot)\|_\infty\leq \|\tilde u^{n-1}_x(t,\cdot)\|_\infty\leq \|u_x(t,\cdot)\|_\infty
\end{equation*}
and 
\begin{equation*}
\| \tilde F^n_x(t,\cdot)\|_\infty\leq \|\tilde F^{n-1}_x(t,\cdot)\|_\infty\leq \|F_x(t,\cdot)\|_\infty.
\end{equation*}
Since $(u_{\Dx}(t,\cdot),F_{\Dx}(t,\cdot))=\p_\Dx(\tilde u^n (t,\cdot), \tilde F^n(t,\cdot))$ for all $t\leq t^n$, we end up with
\begin{equation*}
\sup_{t\in[0,T]}\big(\|u_{\Dx,x}(t,\cdot)\|_\infty+\|F_{\Dx,x}(t,\cdot)\|_\infty\big)\leq \sup_{t\in[0,T]}\big(\|u_x(t,\cdot)\|_\infty+\|F_x(t,\cdot)\|_\infty\big).
\end{equation*}

It is left to show \eqref{est:smooth}. We start by comparing $(u,F)$ with $(\tilde u^0, \tilde F^0)$. To that end let $(t,x)\in [0,T]\times \R$ such that $x_j(t)\leq x\leq x_{j+1}(t)$. Then there exists $x^0$ and $\tilde x^0$ in $[x_j,x_{j+1}]$ such that 
\begin{align*}
u(t,x) &= u_0(x^0)+\frac 12(F_0(x^0)-\frac 12F_\infty)t,\\
	F(t,x) &= F_0(x^0),
\end{align*}
and 
\begin{align*}	
	\tilde u^0(t,x) &= \tilde u^0(0,\tilde x^0)+\frac 12(\tilde F^0(0,\tilde x^0)-\frac 12F_\infty)t\\
		\tilde F^0(t,x) &= \tilde F^0(0,\tilde x^0).
\end{align*}
Using $(\tilde u^0(0,x),\tilde F^0(0,x))= \p_\Dx (u_0(x),F_0(x))$, we have,
\begin{align*}
|u(t,x)-\tilde u^0(t,x)| &\leq |u_0(x^0)-\tilde u^0(0,x^0)| + |\tilde u^0(0,x^0)-\tilde u^0(0,\tilde x^0)|\non &\quad + \frac 12t\left(|F_0(x^0)-\tilde F^0(0,x^0)| + |\tilde F^0(0,x^0)-\tilde F^0(0,\tilde x^0)|\right)\non
		&\leq 2(\|u_{0,x}\|_\infty + \frac 12t\|F_{0,x}\|_\infty)\Dx,
\end{align*}
and
\begin{align*}
	|F(t,x)-\tilde F^0(t,x)| &\leq |F_0(x^0)-\tilde F^0(0,x^0)| + |\tilde F^0(0,x^0)-\tilde F^0(0,\tilde x^0)|\non
		&\leq 2\|F_{0,x}\|_\infty \Dx.
\end{align*}
Combining the last two inequalities, we have
\begin{equation*}
\sup_{t\in[0,T]}\big(\|u(t,\cdot) - \tilde u^0(t,\cdot)\|_\infty + \|F(t,\cdot)-\tilde F^0(t,\cdot)\|_\infty \big)\leq 2L(1+\frac 12t)\Dx,
\end{equation*}
where $L=\sup_{t\in[0,T]} \big(\|u_x(t,\cdot)\|_\infty+\|F_x(t,\cdot)\|_\infty\big)$.
Moreover, we have by the same argument for $t\geq t^n$ that
	\begin{align*}
	\|\tilde u^{n}(t,\cdot) - \tilde u^{n-1}(t,\cdot)\|_\infty & + \|\tilde F^{n}(t,\cdot)-\tilde F^{n-1}(t,\cdot)\|_\infty\\
	& \leq 2(\|\tilde u^{n-1}_x(t^n,\cdot)\|_\infty  + \|\tilde F^{n-1}_x(t^n,\cdot)\|_\infty)(1+\frac 12(t-t^n))\Dx\\
	& \leq 2L(1+\frac12 (t-t^n))\Dx.
	\end{align*}
Since $(u_{\Dx}(t,\cdot),F_{\Dx}(t,\cdot))= \p_\Dx(\tilde u^n(t,\cdot), \tilde F^n(t,\cdot))$ for all $t\leq t^n$, we have for all $t^n\leq t\leq t^{n+1}$ that 
\begin{align*}
&\quad \|u(t,\cdot) -  u_\Dx(t,\cdot)\|_\infty + \|F(t,\cdot)- F_\Dx(t,\cdot)\|_\infty\non 
& \leq  \|u(t,\cdot) - \tilde u_0(t,\cdot)\|_\infty + \|F(t,\cdot)-\tilde F^0(t,\cdot)\|_\infty\non
& \quad + \sum_{l=1}^{n}\left(\|\tilde u^{l}(t,\cdot) - \tilde u^{l-1}(t,\cdot)\|_\infty + \|\tilde F^{l}(t,\cdot)-\tilde F^{l-1}(t,\cdot)\|_\infty\right)\non &\quad + \|\tilde u^n(t,\cdot)-u_\Dx(t,\cdot)\|_\infty+\|\tilde F^n(t,\cdot)-F_\Dx(t,\cdot)\|_\infty\non
		& \leq 2L(1+\frac12 t)\Dx+\sum_{l=1}^{n} 2L(1+\frac12 (t-t^{l}))\Dx +2L\Dx\non
		&\leq C(\sqrt{\Dx}+\Dx),
	\end{align*}
	where we have used \eqref{eq:cfl2}.
\end{proof}

\section{Numerical examples}\label{sec:numex}
\noindent In this section we perform two experiments to see whether the scheme in Definition \ref{def:Godunov num sol} converges to the desired conservative solution. We compare the numerical solutions with known, exact solutions in two cases, namely a peakon and a cusp. These two have been selected since the exact solutions represent two distinct challenges for the numerical solver. The peakon solution experiences wave breaking at $t=2$ and all energy is concentrated in a single point. Thus $F$ becomes discontinuous while $u$ becomes a constant function. The cusp solution, on the other hand, experiences wave breaking at each time $t$ with $t\in [0,3]$, but only an infinitesimal amount of energy concentrates at any given time. 

In our examples we assume that the initial data $u_0$ is constant outside some finite interval $[a,b]$. By \eqref{eq:characteristics} and Lemma~\ref{lemma:basic}, we then obtain that at each time $t$, both $u(t,\cdot)$ and $u_{\Dx}(t,\cdot)$ will be constant outside some finite interval $[a(t),b(t)]$. Thus for any $T>0$, by choosing the computational domain accordingly, we can ensure that $u(t,\cdot)$ and $u_{\Dx}(t,\cdot)$ are constant outside the computational domain for all $t\in [0,T]$. We look at the peakon example first.
\begin{example}[Peakon solution]
	We have initial data
	\begin{subequations}
	\begin{align*}
	u_0(x) &=
		\begin{cases}
		1, &x<0,\\
		1-x, &0\leq x\leq 1,\\
		0, &1<x,
		\end{cases}\\
	F_0(x) &=
		\begin{cases}
		0, &x<0,\\
		x, &0\leq x\leq 1,\\
		1, &1<x,
		\end{cases}
	\end{align*}
	\end{subequations}
	with the exact solution
	\begin{subequations}
	\begin{align*}
	u(t,x) &=
		\begin{cases}
		1-\frac14t, &x< t-\frac 18t^2,\\
		-\frac{1}{1-\frac12t}\left(x-t+\frac18t^2\right)+1-\frac14t, &t-\frac 18t^2\leq x\leq 1+\frac 18t^2<x,\\
		\frac 14t, &1+\frac 18t^2<x,
		\end{cases}\\
	F(t,x) &=
		\begin{cases}
		0, &x<t-\frac 18t^2,\\
		\frac{1}{(1-\frac12t)^2}\left(x-t+\frac18t^2\right), &t-\frac 18t^2\leq x\leq 1+\frac 18t^2,\\
		1, &1+\frac 18t^2<x.
		\end{cases}
	\end{align*}
	\end{subequations}
\end{example}
In Figure \ref{fig:peakon} the numerical solution $(u_{\Delta x},F_{\Delta x})$ is computed and compared with the exact solution at $t=0$, $t=2$, and $t=4$. Figure \ref{fig:peakon error} shows the error, when compared to the exact solution, and we see that the method captures the wave breaking phenomena in this example. 

\begin{figure}[ht]
	\centering
	\subfigure{
	\includegraphics[width=0.38\textwidth]{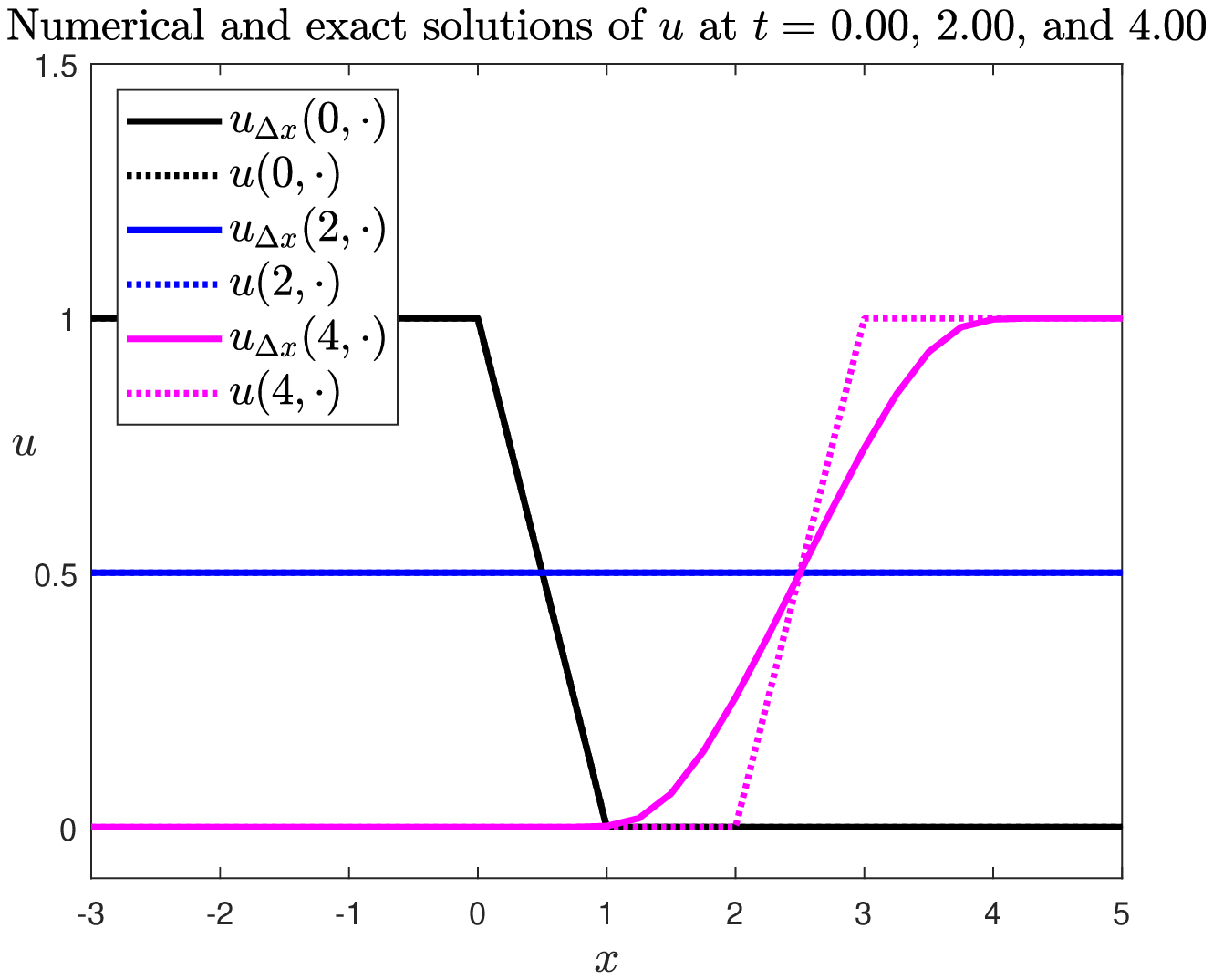}}\hspace{0.5em}
	\subfigure{
	\includegraphics[width=0.38\textwidth]{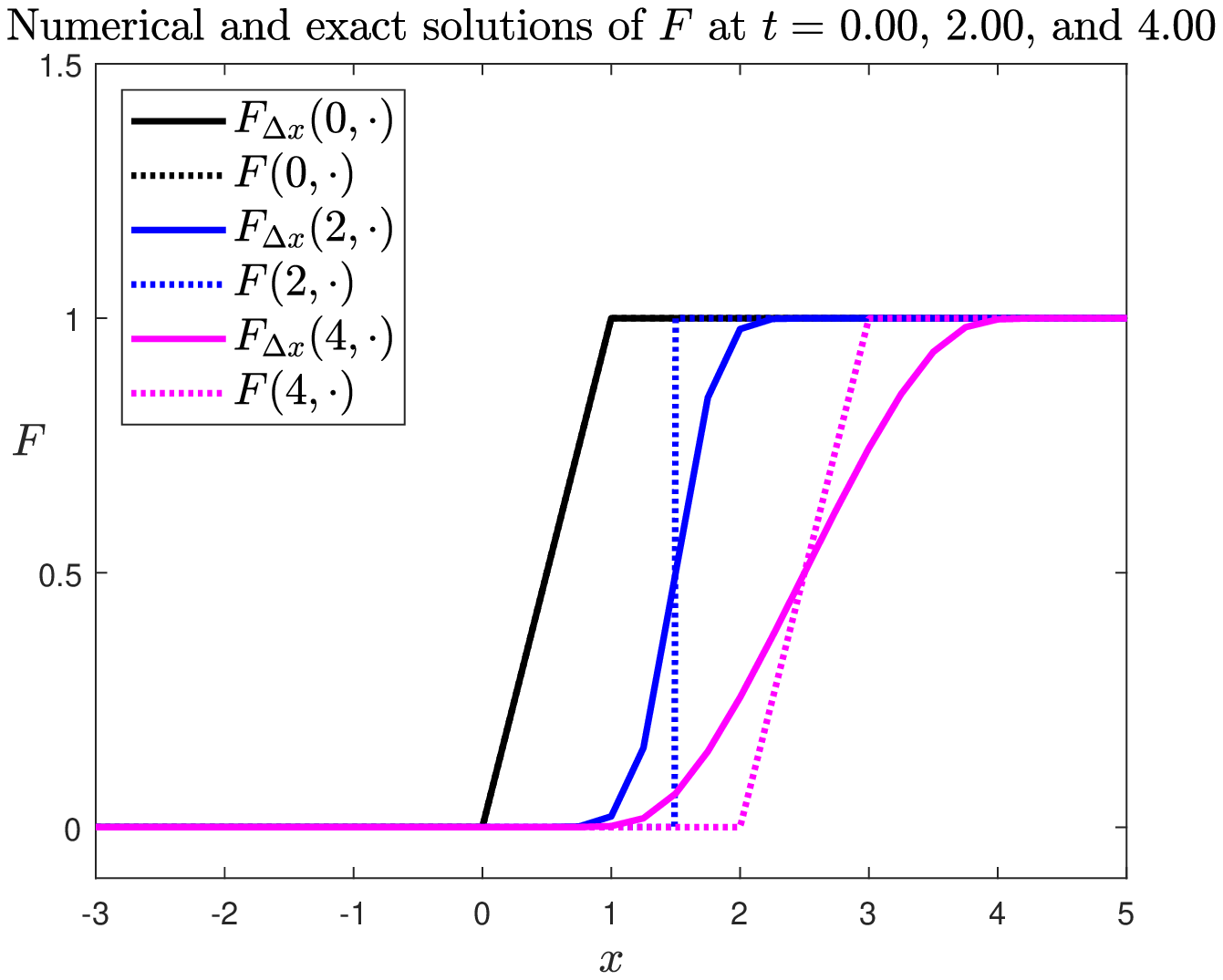}}
	\caption{The functions $u_{\Dx}$ (left) and $F_{\Dx}$ (right) in the case of peakon initial data, plotted at $t=0$, $t=2$, and $t=4$. Here $\Dx = \frac14$ and $\Dt = \frac{1}{4}$.}
	\label{fig:peakon}
\end{figure}
\begin{figure}[ht!]
	\centering
	\subfigure{
	\includegraphics[width=0.38\textwidth]{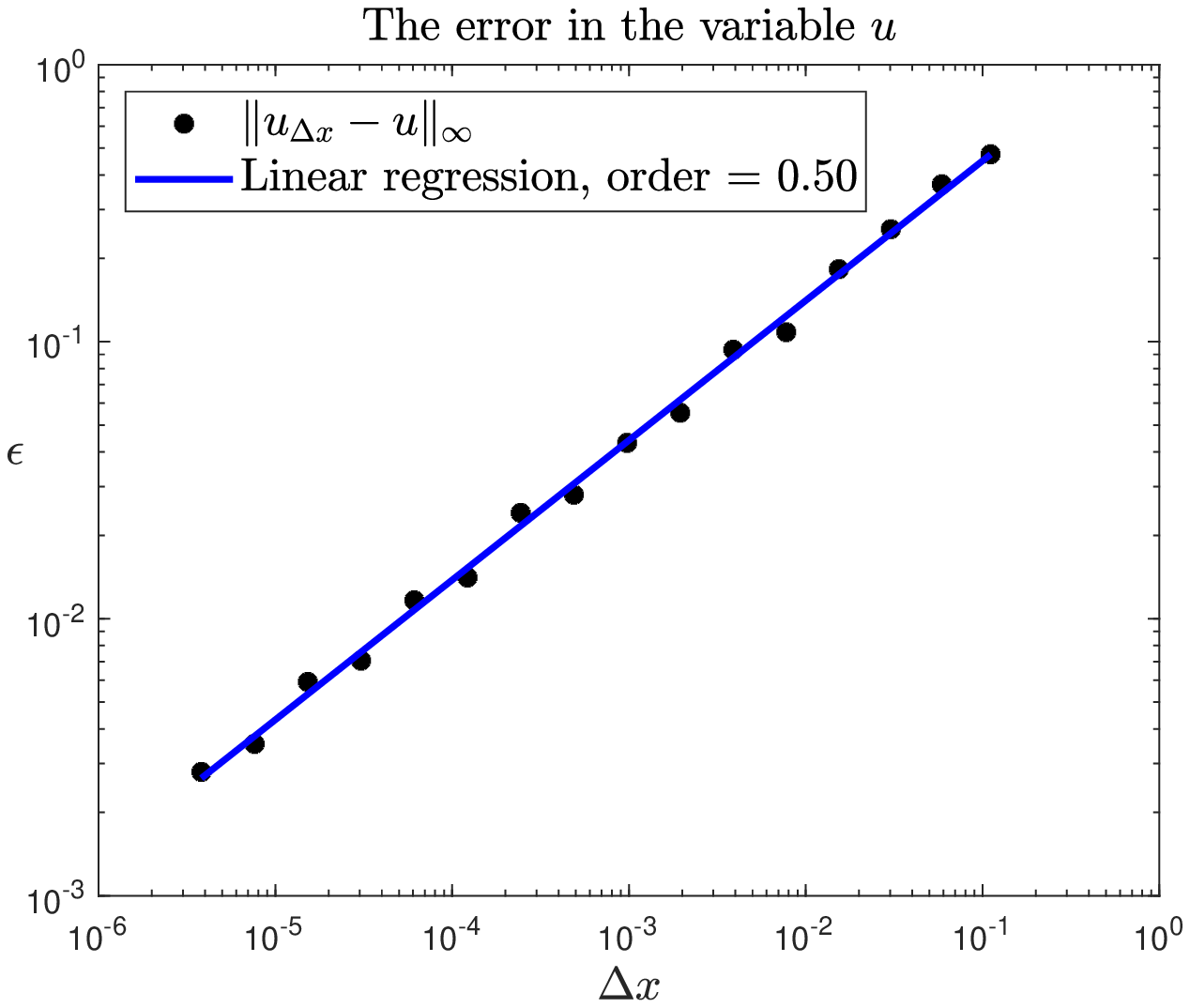}}\hspace{0.5em}
	\subfigure{
	\includegraphics[width=0.38\textwidth]{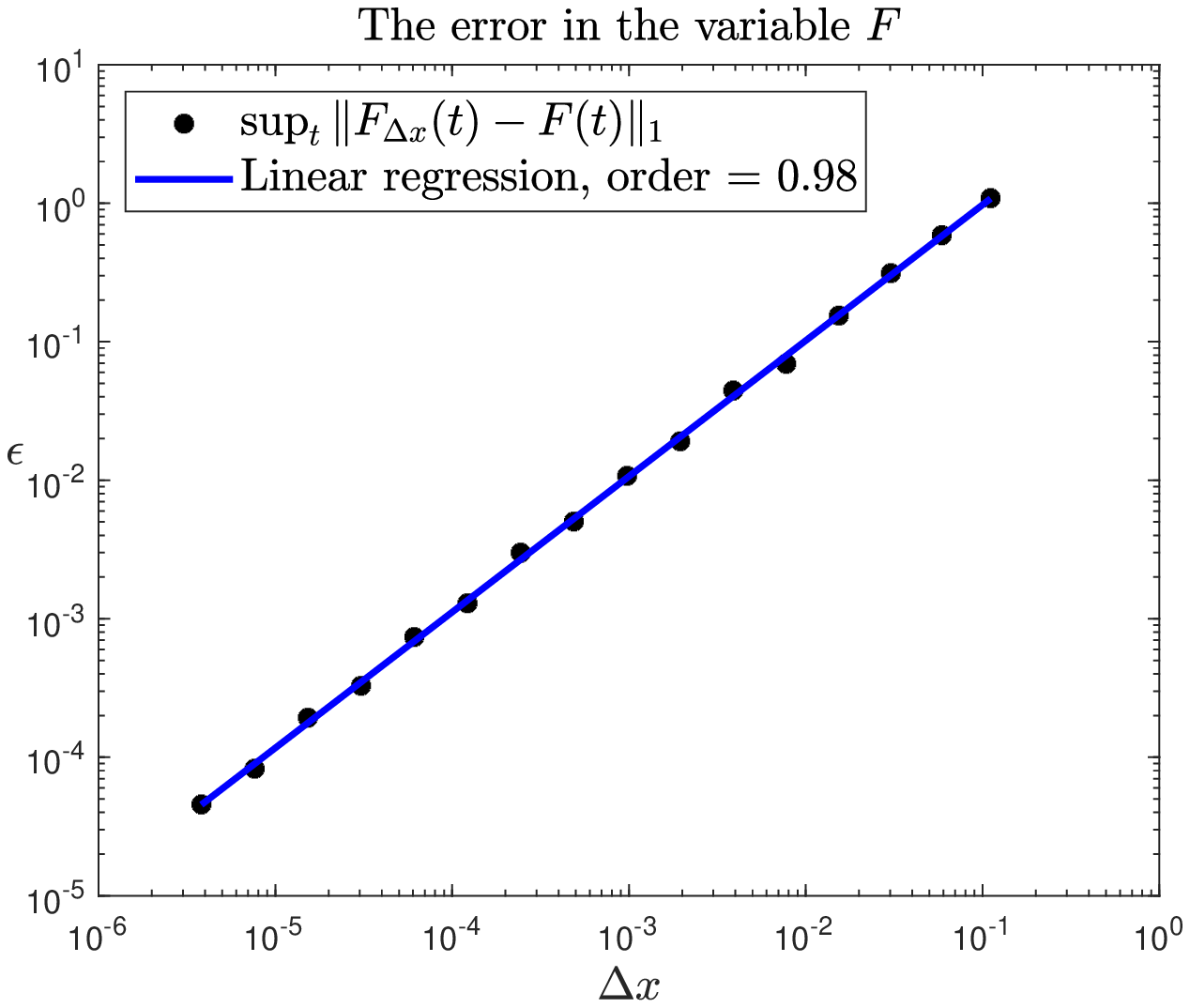}}
	\caption{The $L^\infty$-error of the numerical solution $u_{\Delta x}$ plotted against the spatial grid size $\Delta x$ (left), and the $L^1$-error of the numerical solution $F_{\Delta x}$ plotted against the spatial grid size $\Delta x$ (right).}
	\label{fig:peakon error}
\end{figure}

\begin{example}[Cusp solution]\label{cusp}
	We have initial data
	\begin{subequations}
	\begin{align*}
	u_0(x) &=
		\begin{cases}
		1, &x<-1,\\
		|x|^{\frac 23}, &-1\leq x\leq 1,\\
		1, &1<x,
		\end{cases}\\
	F_0(x) &=
		\begin{cases}
		0, &x<-1,\\
		\frac43\left(x^{\frac 13}+1\right), &-1\leq x\leq 1,\\
		\frac 83, &1<x,
		\end{cases}
	\end{align*}
	\end{subequations}
	with the exact solution
	\begin{subequations}
	\begin{align*}
	u(t,x) &=
		\begin{cases}
		1-\frac23t, &x< -1+t-\frac 13t^2,\\
		\left(x+\left(\frac t3\right)^3\right)^{\frac 23}-\frac{t^2}{9}, &-1+t-\frac 13t^2\leq x\leq 1+t+\frac 13t^2,\\
		1+\frac23t, &1+t+\frac 13t^2<x,
		\end{cases}\\
	F(t,x) &=
		\begin{cases}
		0, &x< -1+t-\frac 13t^2,\\
		\frac43\left(x+\left(\frac t3\right)^3\right)^{\frac 13}+\frac43\left(1-\frac t3\right), &-1+t-\frac 13t^2\leq x\leq 1+t+\frac 13t^2,\\
		\frac83, &1+t+\frac 13t^2<x.
		\end{cases}
	\end{align*}
	\end{subequations}
\end{example}
In Figure \ref{fig:cusp} the numerical solution $(u_{\Delta x},F_{\Delta x})$ is computed and compared with the exact solution at $t=0$, $t=2$, and $t=4$. Figure \ref{fig:cusp error} shows the error when compared to the exact solution. 

From Figure \ref{fig:peakon error} and \ref{fig:cusp error}, we see that the best we can hope for in terms of convergence rates in $L^\infty$ for $u$ and $L^1$ for $F$ in the general case is $\mathcal{O}\big(\sqrt{\Dx} \big)$. As uniqueness of conservative solutions is still an open problem, proving any form of convergence rate of the numerical method seems to be extremely challenging.

\begin{remark}
	Clearly we cannot expect a better convergence order than one half in the $L^\infty$-norm for $u$, since there exists $u_0 \in \mathcal{D}$ such that $\|u_0-u_{0\Delta x}\|_\infty = \sqrt{F_\infty}\sqrt{\Dx}$, see Proposition \ref{prop:proj est}.
\end{remark}

\begin{figure}
	\centering
	\subfigure{
	\includegraphics[width=0.38\textwidth]{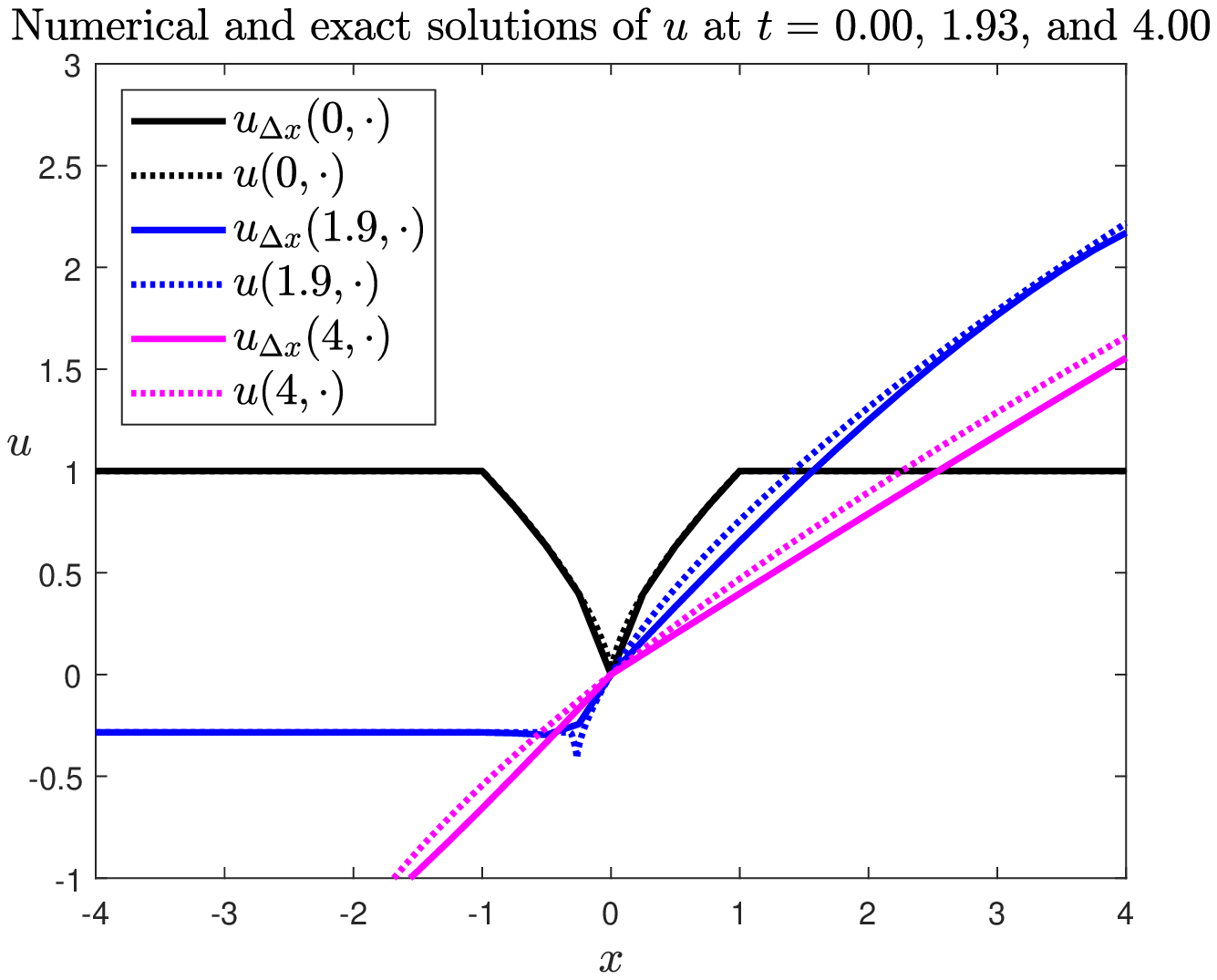}}\hspace{0.5em}
	\subfigure{	
	\includegraphics[width=0.38\textwidth]{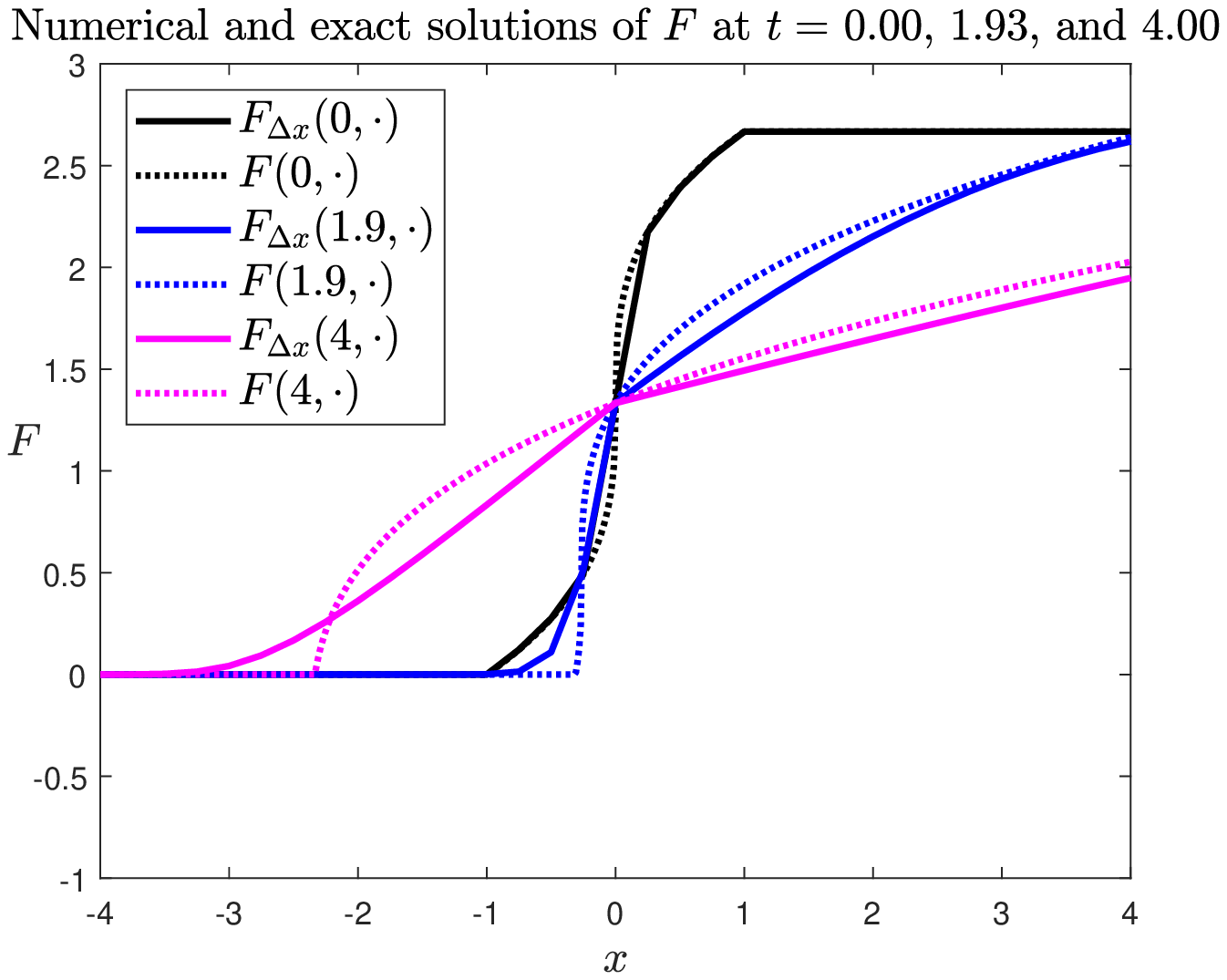}}
	\caption{The functions $u_{\Dx}$ (left) and $F_{\Dx}$, (right) in the case of cusp initial data, plotted at $t=0$, $t=1.93$, and $t=4$. Here $\Dx = 1/4$ and $\Dt \approx 0.148$. Note the slight discrepancy between the numerical solution and the exact solution in the variable $F$ already at $t=0$ due to the projection operator being applied to the numerical initial data.}
	\label{fig:cusp}
\end{figure}

\begin{figure}[ht]
	\centering
	\subfigure{
	\includegraphics[width=0.38\textwidth]{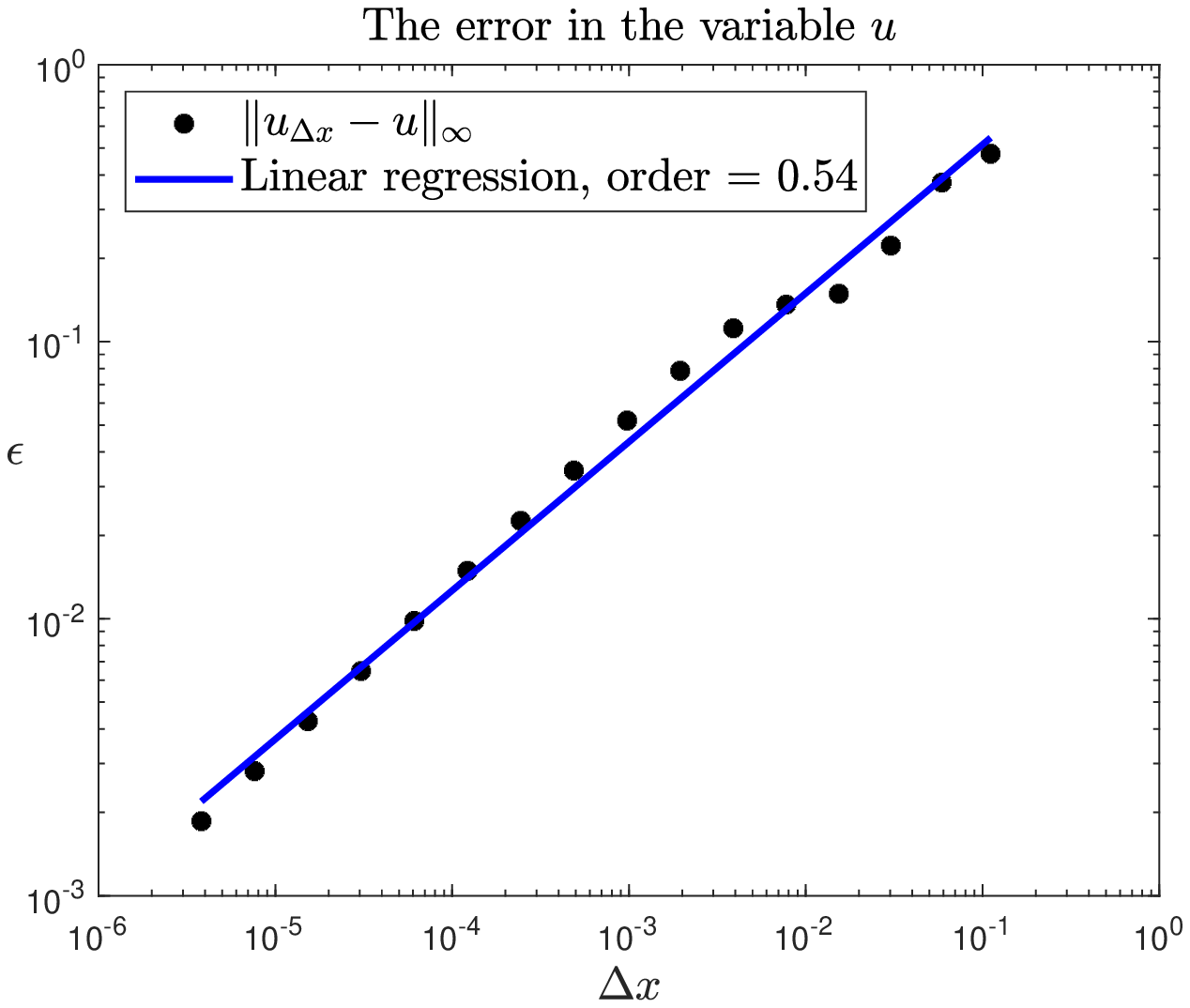}}\hspace{0.5em}
	\subfigure{
	\includegraphics[width=0.38\textwidth]{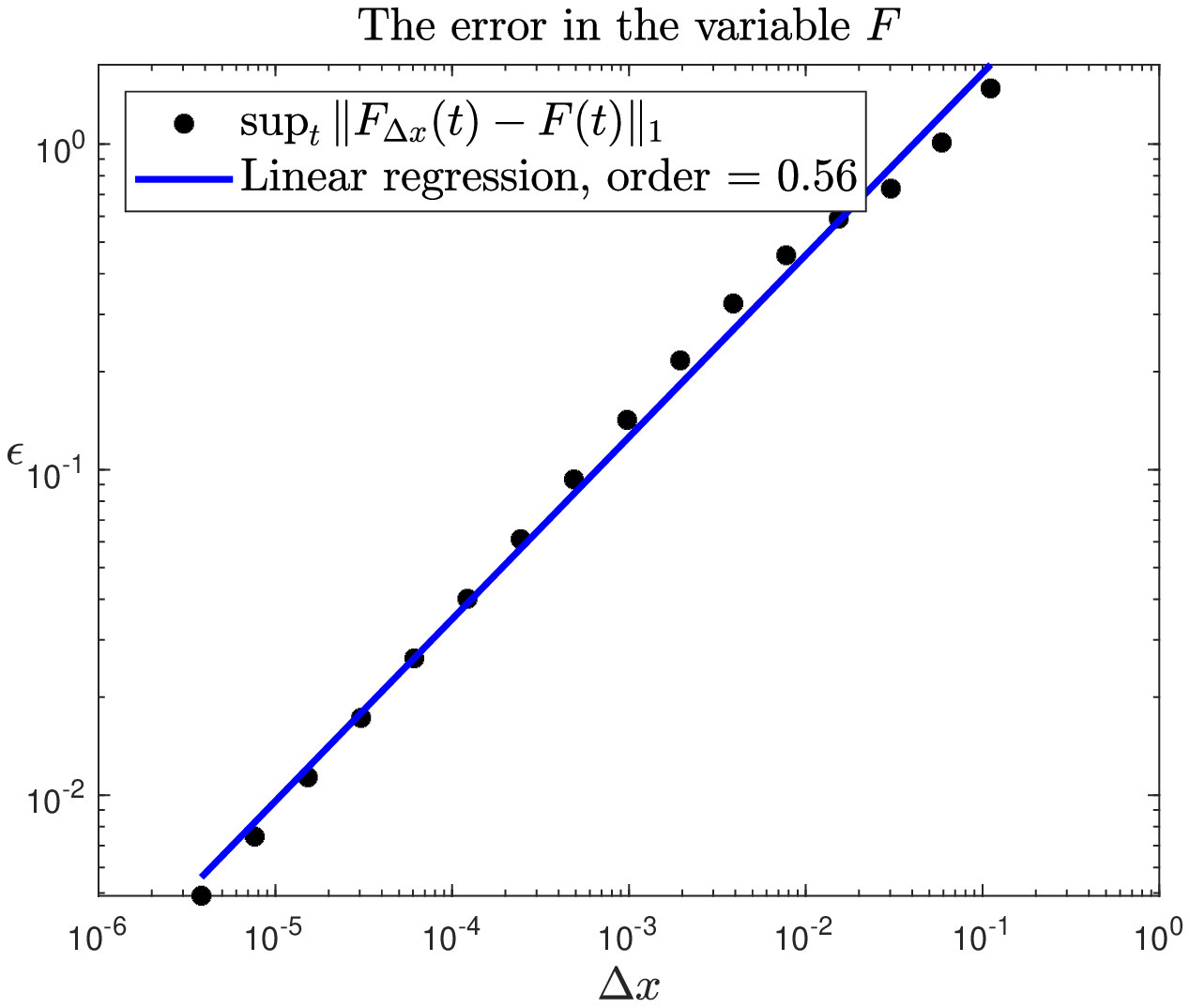}}
	\caption{The $L^\infty$-error of the numerical solution $u_\Dx$ plotted against the spatial grid size $\Dx$ (left) and the $L^1$-error of the numerical solution $F_\Dx$ plotted against the spatial grid size $\Dx$ (right).}
	\label{fig:cusp error}
\end{figure}


\begin{thebibliography}{}

\bibitem{BC}
A. Bressan and A. Constantin, Global solutions of the Hunter--Saxton equation, SIAM J. Math. Anal., 37,996--1026 (2005).

\bibitem{BHR}
A. Bressan, H. Holden, and X. Raynaud, Lipschitz metric for the Hunter--Saxton equation, J. Math. Pures Appl., 94, 68--92 (2010).

\bibitem{BZZ}
A. Bressan, P. Zhang, and Y. Zheng, Asymptotic variational wave equations, Arch. Ration. Mech. Anal., 183, 163--185 (2007).


\bibitem{CH}
 R. Camassa and D. D. Holm, An integrable shallow water equation with peaked solitons, Phys. Rev. Lett., 71, 1661--1664 (1993).

\bibitem{CJ}
T. Cie{\'s}lak and G. Jamr{\'o}z, Maximal dissipation in Hunter--Saxton equation for bounded energy initial data, Adv. Math., 290, 590--613 (2016).

\bibitem{DP}
H.-H. Dai and M. V. Pavlov, Transformations for the Camassa--Holm equation, its high-frequency limit and the Sinh--Gordon equation, J. Phys. Soc. Japan, 67 3655--3657 (1998).

\bibitem{D}
C. M. Dafermos, Generalized characteristics and the Hunter--Saxton equation, J. Hyperbolic Differ. Equ., 8, 159--168, (2011).

\bibitem{folland}
G. B. Folland, Real Analysis. Modern techniques and their applications. Second edition, Pure and Applied Mathematics. A Wiley- Interscience Publication, John Wiley \& Sons, Inc., New York (1999).

\bibitem{GN}
 K. Grunert and A. Nordli, Existence and Lipschitz stability for $\alpha$-dissipative solutions of the two-component Hunter--Saxton system, J. Hyperbolic Differ. Equ., 15, 559--597 (2018).

\bibitem{HH}
H. Hanche-Olsen and H. Holden, The Kolmogorov-Riesz compactness theorem, Expo. Math., 28, 385--394 (2010).

\bibitem{HKR}
H. Holden, K. H. Karlsen, and N. H. Risebro, Convergent difference schemes for the Hunter--Saxton equation, Math. Comp., 76, 699--744 (2007).

\bibitem{HR}
H. Holden and N.H. Risebro, Front tracking for hyperbolic conservation laws, Second edition, Applied Mathematical Sciences, 152, Springer, Heidelberg (2015).

\bibitem{HS}
J. K. Hunter and R. Saxton, Dynamics of director fields, SIAM J. Appl. Math., 56, 1498--1521 (1991).

\bibitem{HZ}
J. K. Hunter and Y. Zheng, On a completely integrable nonlinear hyperbolic variational
equation, Phys. D., 79, 361--386 (1994).

\bibitem{HZ2}
J. K. Hunter and Y. Zheng, On a nonlinear hyperbolic variational equation. I. Global existence of weak solutions, Arch. Ration. Mech. Anal., 129, 305--353 (1995).

\bibitem{HZ3}
J. K. Hunter and Y. Zheng, On a nonlinear hyperbolic variational equation. II. The zero-viscosity and dispersion limits, Arch. Ration. Mech. Anal., 129, 355--383  (1995).

\bibitem{KM}
B. Khesin and G. Misio\l ek, Euler equations on homogeneous spaces and Virasoro orbits, Adv. Math., 176, 116--144 (2003).

\bibitem{MCFM}
Y. Miyatake, D. Cohen, D. Furihata, and T. Matsuo, Geometrical numerical integrator for Hunter--Saxton-like equations, Japan J. Indust. Appl. Math., 34, 441--472 (2017).

\bibitem{N}
A. Nordli, A Lipschitz metric for conservative solutions of the two-component Hunter--Saxton system, Methods Appl. Anal., 23, 215--232 (2016).

\bibitem{ST}
S. Sato, Stability and convergence of a conservative finite difference scheme for the modified Hunter-Saxton equation, BIT, 59, 213--241 (2019).

\bibitem{SX}
R. Saxton, Dynamic instability of the liquid crystal director, In W. B. Lindquist (ed.), Current Progress in Hyperbolic Systems, 100, 325--330 (1989).

\bibitem{XS}
Y. Xu and C.-W. Shu, Local discontinuous Galerkin method for the Hunter-Saxton equation and its zero-viscosity and zero-dispersion limits, SIAM J. Sc. Comput., 31, 1249--1268 (2009).

\bibitem{XS2}
Y. Xu and C.-W. Shu, Dissipative numerical methods for the Hunter--Saxton equation, J. Comput. Math., 28, 606--620 (2010).

\bibitem{Y}
Z. Yin, On the structure of solutions to the periodic Hunter--Saxton equation, SIAM J. Math. Anal., 36, 272--283 (2004).

\bibitem{ZZ}
P. Zhang and Y. Zheng, On the existence and uniqueness of solutions to an asymptotic equation of a variational wave equation, Acta Math. Sin. (Engl. Ser.), 15, 115--130 (1999).

\bibitem{ZZ2}
P. Zhang and Y. Zheng, Existence and uniqueness of solutions of an asymptotic equation arising from a variational wave equation with general data, Arch. Ration. Mech. Anal., 155, 49--83 (2000).


\end{thebibliography}
\end{document}